\documentclass[a4paper,12pt]{article}
\usepackage{amsthm,times}
\usepackage{amssymb,amsmath,tikz}
\usepackage[active]{srcltx}

\def\sym{\mathbb}
\newtheorem{thm}{Theorem}
\newtheorem{lem}{Lemma}
\newtheorem{prop}{Proposition}

\newtheorem{thmabc}{Theorem}

\title{ Voronoi summation formulae and multiplicative functions on
 permutations}
\author{Vytas Zacharovas\footnote{This paper is based on  chapters 1.1 and 1.2 of author's doctorial dissertation that was written and defended at
Vilnius University
in 2004 under supervision of prof. E. Manstavi\v{c}ius. The final version of the paper was written during
the authors work at Institute of Statistical Science, Academia Sinica (Taiwan).}
\\
    Department of Mathematics and Informatics\\
    Vilnius University\\
    Naugarduko 24,  Vilnius\\
    Lithuania
    }

\begin{document}

\maketitle

\begin{abstract}
We prove a Tauberian theorem for the Voronoi summation method of
divergent series with an estimate of the remainder term. The results
on the Voronoi summability are then applied to analyze the mean
values of multiplicative functions on random permutations.
\end{abstract}

\noindent \emph{Key words}: Tauberian theorems, divergent series,
Voronoi summability, N\"{o}rlund summability, symmetric group,
random permutations, additive functions, multiplicative functions,
Berry-Esseen bound.

\section {Introduction}
\subsection{Tauberian theorem for Voronoi summation method}
The classical result of Abel states that if  an infinite series $\sum_{n=0}^\infty a_n$  converges and its sum is equal to $A$ then there exists the limit
\begin{equation}
\label{limit}
\lim_{x\nearrow 1} \sum_{n=0}^\infty a_nx^n=A.
\end{equation}
The converse statement is not true, as can be seen by considering the  series $\sum_{n=0}^\infty (-1)^n$, which is divergent in spite of   existence of  the limit    \(
\lim_{x\nearrow 1} \sum_{n=0}^\infty (-1)^nx^n=1/2
\).
Tauber \cite{tauber} in 1897 proved that if in addition to existence of the limit (\ref{limit}) the coefficients of the infinite series  satisfy condition
\begin{equation}
\label{tauber's_sum}
\sum_{k=0}^n ka_k=o(n)
\end{equation}
as $n\to \infty$, then the series $\sum_{n=0}^\infty a_n$ converges and its sum is equal to $A$. Moreover
the  Tauber's condition (\ref{tauber's_sum}) is in fact necessary for the convergence of the series. This Tauber's result has given rise to the whole class of so called Tauberian theorems. See book \cite{korevaar_tauberian_book} for the review of the subject.

It can be  shown that Tauber's condition
(\ref{tauber's_sum}) imposed on coefficients $a_j$ alone is enough
to provide an asymptotic estimate for partial sums of the formal
series
\[
\sum_{k=0}^n a_k=g(e^{-1/n})+o(1),
\]
where $g(z)=\sum_{j=0}^\infty a_jz^j$. Note that here $g(e^{-1/n})$
does not need to have a limit as $n\to \infty$.

Voronoi (the same summation method has been later reintroduced by
N\"{o}rlund and is often named him) introduced a summation method of
divergent series which is defined by a sequence of non-negative
numbers $w_j\geqslant 0$, that are not identically equal to zero.
Suppose $\sum_{k=0}^{\infty}a_k$ is a formal series. If there exists
a limit
$$
\lim_{n\to \infty
}{\frac{w_0s_n+w_1s_{n-1}+\cdots+w_ns_0}{w_0+w_1+\cdots +w_n}}=s\in
\sym C,
$$
where $s_k=a_0+a_1+\cdots +a_k$, then we say that the series
$\sum_{k=0}^{\infty}a_k$ can be summed in the sense of Voronoi and
its Voronoi sum is equal to $s$. In such case we write
$$
(W,w_n)\sum_{k=0}^{\infty}a_k=s.
$$
If, for example, we take  $w_0=1$ and $w_j=0$, $j\geqslant 1$ then
the Voronoi summation for such choice of $w_j$  will coincide with
the usual definition of convergence of an infinite series. The
choice $w_j\equiv 1$ leads to the definition of Ces\`{a}ro $(C,1)$
summability.  We refer the reader to the classical book of Hardy
\cite{hardy_divergent} for more examples and discussions on the
subject of divergent series.

Note that the weighted average of partial  sums of the formal series
\(\sum_{j=0}^\infty a_j\) defining Voronoi summation method  can be rewritten
as
\[
\frac{w_0s_n+w_1s_{n-1}+\cdots+w_ns_0}{w_0+w_1+\cdots
+w_n}=\frac{1}{W_n}\sum_{k=0}^na_kW_{n-k},
\]
where $W_j=w_0+w_1+\cdots+w_j$ are positive numbers. In what follows
we will refer to the above weighted average of partial sums $s_j$ as
\emph{Voronoi mean}. Thus given a fixed sequence of positive numbers
$W_j$, a natural question arises, what would be the generalization
of the classical Tauber's condition (\ref{tauber's_sum}) on $a_k$
that would imply the following asymptotic for Voronoi mean
\begin{equation}
\label{voronoi_asymtp}
\frac{1}{W_n}\sum_{k=0}^na_kW_{n-k}=g(e^{-1/n})+o(1),
\end{equation}
as $n\to \infty$, where as before $g(x)=\sum_{k=0}^\infty a_kx^k$?
We provide a partial answer to this question  for the the class of
sequences $W_j$ whose generating function is of the form
\[
\sum_{j=0}^\infty W_jz^j=\exp\left\{\sum_{k=1}^\infty
\frac{u_k}{k}z^k\right\}
\]
with $0<u^- \leqslant u_k\leqslant u^+<\infty$. We will show that for
this class of Voronoi methods, if $a_k$ satisfy condition
\begin{equation}
\label{voronoi_tauberian_condition}
\frac{1}{W_n}\sum_{k=0}^nka_kW_{n-k}=o(n)
\end{equation}
as $n\to \infty$, then the asymptotic (\ref{voronoi_asymtp}) for Voronoi means holds.
Note that the sequence $W_n$ satisfying our condition imposed on the
form of its generating function does not need to be increasing,
unlike the sequences arising from the definition of Voronoi
summation, in which case $W_n$ as a partial sum of $w_j$ should be
increasing $W_{n+1}-W_n=w_{n+1}\geqslant 0$. This class  is large enough to contain the class of Ces\`{a}ro
methods with parameter $\theta>-1$ (see \cite{hardy_divergent} for
definition). An open question remains how far can we expand the
class of Voronoi methods so that the condition
(\ref{voronoi_tauberian_condition}) on $a_n$ would guarantee the
asymptotic (\ref{voronoi_asymtp}) for Voronoi means.

The central part of our argument is the inequality of the following
theorem that allows us to estimate the error term in the asymptotic
of Voronoi means (\ref{voronoi_asymtp}) in terms of sums
$\sum_{k=0}^nka_kW_{n-k}$.

\begin{thm}
\label{fundthm1} Suppose $g(z)=\sum_{n=0}^{\infty}a_nz^n$ is an analytic function for
$|z|<1$ and  $p_j$ is a sequence of positive numbers  that is defined by means
of its generating function
\[
p(z)=\sum_{j=0}^\infty
 p_jz^j=\exp \Bigl\{
 \sum_{k=1}^{\infty} {\frac{d_k}k}z^k \Bigr\},
\]
where $d_j$ are positive numbers bounded from above and below
$0<d^-\leqslant d_j\leqslant d^+$. Then  there
exists a positive constant   $c=c(d^+,d^-)$, which depends on $d^+$ and $d^-$
only, such that for all $n\geqslant 1$ holds the inequality
\begin{equation}
\label{fund_ineq}
\begin{split}
&\left|{\frac{1}{p_n}}\sum_{k=0}^na_kp_{n-k}-g(e^{-1/n})-{\frac{S(g;n)}{np_n}}\right|
\\
&\quad\leqslant c\left( {\frac{1}{n^\theta}}\sum_{j=1}^n
{\frac{|S(g;j)|}{p_j}}j^{\theta -2}+ {\frac{1}{p(e^{-1/n
})}}\sum_{j>n}\frac{|S(g;j)|}{j}e^{-j/n}\right),
\end{split}
\end{equation}
where
\[
S(g;j)=\sum_{k=0}^ja_kkp_{j-k}
\]
and $\theta =\min \{ d^-,1\}$.
\end{thm}
A simple consequence of the above theorem is the direct
generalization of Tauber's theorem for our class of Voronoi summation methods.
\begin{thm}
\label{voronmean} Let $g(z)$ and $p(z)$ be the same as in Theorem
\ref{fundthm1}. Then the  relation
$$
\lim_{n \to \infty}{\frac{1}{p_n}}\sum_{k=0}^{n}a_kp_{n-k}=A\in \sym
C
$$
holds if and only if the following two conditions are satisfied:
\begin{enumerate}
\item
\begin{equation}
\label{first_cond}
 \lim_{x\uparrow 1}g(x)= A,
\end{equation}
\item
\begin{equation}
\label{second_cond} \sum_{k=0}^{n}a_kkp_{n-k}=o(np_n) \quad
\hbox{as}\quad n\to \infty .
\end{equation}
\end{enumerate}
\end{thm}
\subsection{Application to analysis of generating functions}
In what follows, for any  analytic function $H(z)=\sum_{j=0}^\infty
H_jz^j$ we will denote by $[z^n]H(z)$ its $n$-th Taylor coefficient
$H_n$. We can apply our results on Voronoi summation method to
analysis of generating functions in the following way. Suppose we
want to analyze the asymptotic behavior of the coefficients
$[z^n]F(z)$ in the Taylor expansion of the generating function
$F(z)$, which is analytic in the unit disc $|z|<1$. Suppose we can
decompose the generating function $F(z)$ as a product of two functions
\[
F(z)=W(z)g(z),
\]
where $W(z)=\sum_{j=0}^\infty W_jz^j$ is a function with positive
Taylor coefficients $W_j>0$ for $j\geqslant 1$ with $W_0=1$ and such
that the coefficients in  the Taylor expansion of its logarithmic
derivative
 are bounded from above and below
\[
0<u^-\leqslant [z^n]\frac{ W'(z)}{W(z)}\leqslant u^+
\]
for all $n\geqslant 0$  by some fixed positive constants
$u^+$,$u^-$. If in addition to the above restrictions on $W(z)$,  our
Tauberian condition (\ref{voronoi_tauberian_condition}), expressed
in terms of  generating functions as
\[
[z^{n-1}]W(z)g'(z)=o(nW_n),
\]
is satisfied, then by the inequality (\ref{fund_ineq}) of Theorem \ref{fundthm1} we get
an estimate for the Taylor coefficients
\[
[z^n]F(z)=[z^{n}]W(z)g(z)=W_n\bigl(g(e^{-1/n})+o(1)\bigr),
\]
as $n\to \infty$.
This approach can be compared with the other standard technique for
analyzing asymptotic behavior of the Taylor coefficients of analytic
functions. It is based on representing $[z^n]F(z)$ as a Cauchy
integral
\begin{equation}
\label{cauchy_integral}
[z^n]F(z)=\frac{1}{2\pi
i}\int_{|z|=r}\frac{F(z)}{z^{n+1}}\,dz.
\end{equation}
The further analysis depends on the amount and the type of
information that we have on the behavior  of $F(z)$ near its
singularities and whether or not function $F(z)$ can be analytically
extended to some area beyond the radius of convergence of its Taylor series.
Flajolet and
Odlyzko \cite{flajolet_odlyzko_1990} analyzed the case when $F(z)$
can be decomposed as $F(z)=W(z)g(z)$ where
$W(z)=\frac1{(1-z)^\theta}=\sum_{n=0}^\infty\binom{n+\theta-1}{n}z^n$
and $g(z)$ is an analytic function in the circle $|z|<
1+\varepsilon$, where $\varepsilon>0$. They proved that
\[
[z^n]\frac{1}{(1-z)^\theta}g(z)=\binom{n+\theta-1}{n}g(1)\bigl(1+O(n^{-1})\bigr).
\]
The analysis of the Cauchy integral (\ref{cauchy_integral}) usually
becomes considerably more difficult if we do not know anything about
the analytic extension of generating function $F(z)$ beyond the unit
circle $|z|<1$. This is exactly the case with the generating
function of the mean values of multiplicative functions on
permutations, which is the main object of application of our theorems
for Voronoi sums. Manstavi\v{c}ius in a series of papers
\cite{manst_berry_1998}, \cite{manst_tauber_1999},
\cite{manst_decomposable_2002}  and \cite{manst_additive_1996} used
a technique based on Hal\'{a}sz'es\cite{halasz_1968} ideas for
investigating asymptotic behavior of such Cauchy integrals. See also
\cite{babu_manst_zakh_2007} for a modified version of this approach.
The approach we use here   exploits the fact that in the case of
random permutations, the generating functions we consider are such
that a simple upper bound can be obtained for the quantity
$[z^{n-1}]W(z)g'(z)$ when $W(z)g(z)$ is the appropriate
decomposition of the generating function,
 thus immediately
leading to the asymptotic of type
$[z^n]W(z)g(z)=W_n\bigl(g(e^{-1/n})+o(1)\bigr)$. The advantage of
such an approach is that it allows us to avoid the analysis of
function $F(z)$ for complex values of $z$, which is particulary hard
to do since such analysis requires estimating certain complicated
trigonometric sums. We only use the information on the behavior of
$g(x)$ for the real values $x$ that are close to $1$.  The same
approach has already been used in our papers \cite{zakh_cesaro_2001}
and \cite{zakh_palanga_2001} to analyze the distribution of additive
and multiplicative functions with respect to Ewens measure.

\subsection{Random permutations}
Let $S_n$ be the symmetric group. Recall that $S_n$ is composed of
all possible functions that bijectively map the set of first $n$
integers $\{1,2,\ldots,n\}$ into itself. Such functions are also
called permutations. Every permutation $\sigma$ belonging to the
symmetric group $S_n$ can be represented as an oriented graph,
containing $n$ vertices that are labeled by natural numbers
$1,2,\ldots, n$, and $n$ edges, each edge corresponding to a pair
$(j,\sigma(j))$, starting at vertex $j$ and pointing to vertex
$\sigma(j)$. Such graphs are characterized by the property that
every edge has one and only one outgoing edge and one and only one
incoming edge. It is easy to realize that each such a graph consists
only of cyclical components. For example, let us consider
permutation
\begin{equation}
\label{example_of_sigma} \lambda=\left(
\begin{array}{ccccccccccccccc}
1 & 2 & 3 & 4 & 5 & 6 & 7 & 8 & 9 & 10 & 11 & 12 & 13 & 14 & 15
\\
4 & 2 & 9 & 6 & 7 & 10 & 8 & 5 & 3 & 1 & 11 & 12 & 15 & 13 & 14
\end{array} \right)
\end{equation}
belonging to $S_{15}$, written in its usual representation as a
table consisting of two rows. The upper row contains the numbers
$1,2,\ldots,15$ and the lower row consists of their images
$\lambda(1),\lambda(2),\ldots,\lambda(15)$. Such permutation
corresponds to the graph consisting of seven cyclical components.

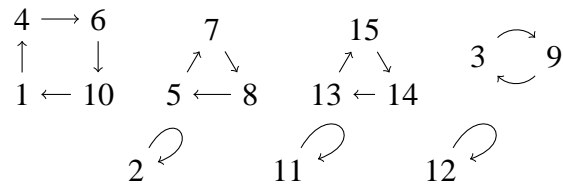
\begin{center}
\begin{tikzpicture}
 \node (1) at (0,0){1};
 \node (2) at(0,1){4};
 \node (3) at (1,1){6};
 \node (4) at (1,0){10};

 \node (5) at (2,0){5};
 \node (6) at (2.5,0.86){7};
 \node (7) at (3,0){8};

 \node (8) at (6,0.5){3};
 \node (9) at (7,0.5){9};

 \node (10) at (1.5,-1){2}edge[in=20,out=60,loop]();
 \node (11) at (3.5,-1){11}edge[in=20,out=60,loop]();
 \node (15) at (5.5,-1){12}edge[in=20,out=60,loop]();

 \node (12) at (4,0){13};
 \node (13) at (4.5,0.86){15};
 \node (14) at (5,0){14};

 \draw[->] (1) to (2);
 \draw[->] (2) to (3);
 \draw[->] (3) to (4);
 \draw[->] (4) to (1);
 \draw[->] (5) to (6);
 \draw[->] (6) to (7);
 \draw[->] (7) to (5);
 \draw[->] (8) to [bend left=45](9);
 \draw[->] (9) to [bend left=45](8);
 \draw[->] (12) to (13);
 \draw[->] (13) to (14);
 \draw[->] (14) to (12);

\end{tikzpicture}
\end{center}
Following \cite{manst_additive_1996} we will consider the classes of
additive and multiplicative  functions on permutations whose values
are determined by the decomposition of permutations into cyclical
components. These functions are defined as follows. Suppose we have
$n$ real numbers $\hat{h}(1),\hat{h}(2),\ldots,\hat{h}(n)$, then for
each permutation $\sigma\in S_n$ we can compose a sum $h(\sigma)$
over all cycles in the graph of $\sigma$ so that  for each cycle of
length $j$ we add one summand $\hat{h}(j)$. Or equivalently,
\[
h(\sigma)=\hat{h}(1)\alpha_1(\sigma)+\hat{h}(2)\alpha_2(\sigma)+\cdots+\hat{h}(n)\alpha_n(\sigma),
\]
where $\alpha_j(\sigma)$ is the number of cycles of length $j$ in
permutation $\sigma$. For example $15$ numbers
$\hat{h}(1),\hat{h}(2),\ldots,\hat{h}(15)$ will completely define
the value of additive function $h(\sigma)$ on all $\sigma \in
S_{15}$. In particular, our permutation $\lambda\in S_{15}$ contains
one cycle of length $4$, two cycles of length $3$, one cycle of
length $2$ and three cycles of length $1$, therefore
\[
h(\lambda)=3\hat{h}(1)+\hat{h}(2)+2\hat{h}(3)+\hat{h}(4).
\]
A simple example of an additive function is obtained if all
$\hat{h}(j)$  are equal to $1$.  The resulting additive function
$w(\sigma)$ is then equal to the total number of cycles of in the graph
of $\sigma$. For our example (\ref{example_of_sigma}) we have
$w(\lambda)=7$. Goncharov \cite{goncharov} was the first to analyze
the limit distribution of $w(\sigma)$. He proved that the number of
permutations $\sigma \in S_n$ satisfying inequality
$\frac{w(\sigma)-\log n}{\sqrt{\log n}}<x$ divided by the total
number of permutations $|S_n|=n!$ converges to
$\frac{1}{\sqrt{2\pi}}\int_{-\infty}^xe^{-u^2/2}\,du$ as $n\to
\infty$.

In a similar way we define multiplicative functions on the symmetric
group $S_n$. Suppose we are given $n$ complex numbers
$\hat{f}(1),\hat{f}(2),\ldots,\hat{f}(n)$. Then for each permutation
$\sigma\in S_n$ we can assign a product $f(\sigma)$ over all cycles
belonging to the oriented graph
 of $\sigma$ that contains one multiplicand $\hat{f}(j)$ corresponding to
 every cycle of size $j$ belonging to $\sigma$. In other words
\[
f(\sigma)=\hat{f}(1)^{\alpha_1(\sigma)}\hat{f}(2)^{\alpha_2(\sigma)}\cdots\hat{f}(n)^{\alpha_n(\sigma)},
\]
where we assume $0^0=1$. For our example  (\ref{example_of_sigma})
of $\lambda\in S_{15}$ we have
\[
f(\lambda)=\hat{f}(1)^{3}\hat{f}(2)\hat{f}(3)^2\hat{f}(4).
\]

Suppose $d(\sigma )$ is a non-negative multiplicative function,
which is not identically equal to zero. Then we can define
 a probabilistic measure $\nu_{n,d}$ on $S_n$ by the formula
\begin{equation}
\label{meas} \nu_{n,d}(\sigma)=\frac{d(\sigma )}{\sum_{\tau \in S_n}
d(\tau )}.
\end{equation}
The simplest and the most natural choice is to put $\hat
d(j)\equiv1$, which leads to the uniform probability measure
\[
\nu_{n}^{(1)}(\sigma)=\frac{1}{n!}.
\]
Thus Goncharov's result can be expressed in probabilistic terms as a
limit theorem
\[
\nu_{n}^{(1)}\left(\frac{w(\sigma)-\log n}{\sqrt{\log
n}}<x\right)\to
\frac{1}{\sqrt{2\pi}}\int_{-\infty}^xe^{-u^2/2}\,du,\quad
\hbox{as}\quad n\to \infty,
\]
stating that the number of cycles $w(\sigma)$ in permutation
$\sigma$ chosen with equal probability among all the permutations of
the symmetric group $S_n$ is asymptotically normally distributed.

 More generally if all $\hat d(j)$ are equal $\hat d(j)\equiv
\theta
>0$ then $d(\sigma)=\theta^{w(\sigma)}$, thus we obtain the so called
Ewens probability measure
\[
\nu_{n}^{(\theta)}(\sigma)=\frac{\theta^{w(\sigma)}}{\sum_{\tau\in S_n}\theta^{\omega(\tau)}}=\frac{\theta^{w(\sigma)}}{\theta(\theta+1)\cdots
(\theta+n-1)}.
\]
Let us denote by $M_n^d(f)$ a weighted mean of a
multiplicative function $f:S_n \to \sym C$ with respect to the
measure $\nu_{n,d}(\sigma)$:
$$
M_n^d(f)=\sum_{\sigma \in S_n}f(\sigma )\nu_{n,d}(\sigma )=
{\frac{\sum_{\sigma \in S_n}f(\sigma )d(\sigma )}{\sum_{\sigma \in
S_n}d(\sigma )}}.
$$
In 2002 Manstavi\v cius proved the following result.
\begin{thmabc}[\cite{manst_decomposable_2002}]
\label{t_manstdecomp} Let $f:S_n \to \sym C$ be a multiplicative
function, such that $|f(\sigma )|\leqslant 1$, satisfying the
conditions:
\begin{equation}
\label{bound} \sum_{j\leqslant n}{\frac{1 -\Re \hat
f(j)}{j}}\leqslant D
\end{equation}
and
$$
\frac{1}{n}\sum_{j=1}^n |\hat f (j) -1| \leqslant \mu_n =o(1),
$$
for some positive constant $D$ and some sequence $\mu_n$.

Suppose that the measure defining multiplicative
function $d(\sigma)$ satisfies the condition
 $0<d^-\leqslant \hat{d}(j)\leqslant d^+$ for all $j\geqslant 1$,  with some fixed positive
constants $d^-$ and $d^+$, then there exist  positive constants
$c_1=c_1(d^-,d^+)$ and $c_2=c_2(d^-,d^+)$  such that
$$
M_n^d(f)=\exp \left\{ \sum_{j\leqslant n}{\hat{d}(j){\frac{\hat
f(j)-1}{j}}} \right\} +O\left(  \mu_n^{c_1}+{\frac{1}{n^{c_2}}}
\right).
$$
\end{thmabc}
 We prove the following result.
\begin{thm}
\label{meanf} Let $f:S_n\to \sym C$ be a multiplicative function
satisfying the condition $|f(\sigma)|\leqslant 1$ for all $\sigma
\in S_n$. Suppose that the measure defining multiplicative function
$d(\sigma)$ is such that $0<d^-\leqslant \hat{d}(j) \leqslant d^+$.
Then we have
\begin{multline*}
\Delta_n:=\left| M_n^d(f)- \exp \left\{
\sum_{j=1}^{n}\hat{d}(j){\frac{\hat f(j)-1}{j}} \right\} \right|
\\
\leqslant c_1 \left( {\biggl(\sum_{j=0}^np_j
\biggr)}^{-1}\sum_{k=1}^{n}|\hat f(k) -1|p_{n-k}+ {\frac{1}{n^{d^-}
}} \sum_{k=1}^{n}|\hat f(k) -1|k^{d^- - 1}\right.
\\
\left. + {\frac{1}{n}}\sum_{k=1}^{n}|\hat f(k) -1| \right)
\end{multline*}
for $d^-<1$ and
$$ \Delta_n\leqslant c_1 \left( {\biggl(\sum_{j=0}^np_j
\biggr)}^{-1}\sum_{k=1}^{n}|\hat f(k) -1|p_{n-k}+ {\frac{1}{n }}
\sum_{k=1}^{n}|\hat f(k) -1| \left( 1+\log {\frac{n}{k}} \right)
\right)
$$
for $d^-\geqslant 1$,  where $c_1=c_1(d^-,d^+)$ is a positive
constant which depends on $d^-$ and $d^+$ only, and
$$
p_n={\frac{1}{n!}}\sum_{\sigma \in
S_n}d(\sigma)=[z^n]\exp\left\{\sum_{j=1}^\infty
\frac{\hat{d}(j)}{j}z^j\right\}
$$
\end{thm}
Thus Theorem \ref{meanf} shows that  condition (\ref{bound}) in
Theorem A is superfluous. The inequality of our theorem also yields  more accurate
estimate of the remainder term than Theorem \ref{t_manstdecomp}.

Note that if function $h(\sigma)$ is additive then function
$\exp\bigl(it h(\sigma)\bigr)$ is multiplicative which means that
the characteristic function of an additive function with respect to
our measure (\ref{meas}) is a mean value of a multiplicative
function. It follows hence that the estimates for the mean values of
multiplicative functions allow us to obtain information on the
distribution of the values of additive functions.

 Let us denote
\[
A(n)=\sum_{k=1}^n\hat{d}(k){\frac{\hat h_{n}(k)}{k}},
\quad
C_n=\sum_{j=1}^n\hat{d}(j){\frac{\hat h_{n}(j)}{j}}\left( {\frac{p_{n-j}}{p_n}}-1
\right),
\]
and
\[
L_{n,p}=\sum_{k=1}^n{\frac{|\hat h_{n}(k)|^p}{k}},
\quad L_{n,2}'=\sum_{j=1}^n{\frac{\hat h_{n}^2(j)}{j}}\left| {\frac{p_{n-j}}{p_{n}}} -1 \right|.
\]
Henceforth we assume that $\tilde h_{n}(k)$ satisfies the
normalizing condition
\begin{equation}
\label{dnormalize} \sum_{k=1}^n\hat{d}(k) {\frac{\hat
h_{n}^2(k)}{k}}=1.
\end{equation}
\begin{thm}
\label{clt1}
Suppose $0<d^-\leqslant d_j \leqslant d^+$,
and $p$ is a fixed number  such that
 $\infty ~\geqslant~ p>\max \left\{ 2, 1/d^- \right\}$.
Suppose
\[
   F_n(x)=\nu_{n,d} \left(
           h(\sigma )-A(n)<x  \right),
\]
where $h(\sigma )$ is a additive  function satisfying the normalizing
 condition (\ref{dnormalize}). Then  we have
\[
      \sup_{x\in \sym R}\left| F_n(x)-\Phi (x)+{\frac{1}{\sqrt{2\pi}}}e^{-x^2/2}
          C_n\right|\ll L_{n,3}+L_{n,p}^{2/p}+L_{n,2}',
\]
here we assume that
\[
 L_{n,\infty }^{1/\infty}=\lim_{p\to \infty}
  L_{n,p}^{1/p}=\max_{1\leqslant j \leqslant n}|\hat h(j)|
\]
 for $p=\infty$.
 \end{thm}
Theorem \ref{clt1} generalizes the corresponding result of
Manstavi\v{c}ius \cite{manst_berry_1998} that was proved for the case of
uniform measures $\hat{d}(j)\equiv 1$, later generalized for  Ewens
measures $\hat d(j)\equiv \theta >0$ in our paper
\cite{zakh_palanga_2001}.

\section{Proofs}

\subsection{Voronoi summation method}
Throughout the proofs we will routinely use a simple inequality
\begin{equation}
\label{upper_bound_for_sum} b_0+b_1+\cdots +b_n\leqslant
\sum_{k=0}^{n}b_ke^{1-k/n}  \leqslant
\sum_{k=0}^{\infty}b_ke^{1-k/n}= eb(e^{-1/n})
\end{equation}
for partial sums of coefficients of a generating function
$b(x)=\sum_{k=0}^{\infty}b_kx^k$ with nonnegative coefficients
$b_k\geqslant 0$. The next theorem shows that a similar lower bound
is also valid if the logarithmic derivative of the generating
function $b(x)$ does not grow too fast as $x\to1$.
\begin{thm}
\label{lowcoefth} Let  $b(x)=\sum_{k=0}^{\infty}b_kx^k$ be a series
with non-negative $b_k\geqslant 0$ coefficients, that converges in
the interval $x\in[0,1)$ . Suppose there exists such $c>0$ that the logarithmic derivative of $b(x)$
satisfies the inequality
\begin{equation}
\label{logarithmic_derivative}
\frac{b'(x)}{b(x)}\leqslant \frac{c}{1-x},
\end{equation}
for all $0\leqslant x <1$. Then there exists such a positive constant
$K=K(c)$ that
\[
\sum_{k=0}^Nb_k\geqslant K(c) b(e^{-1/N}),
\]
for all $N\geqslant 2c$.
\end{thm}

\begin{proof}
 Suppose $0\leqslant x < 1$ and $N\geqslant 2c$, then
\begin{eqnarray*}
b(x)&\leqslant& \sum_{k=0}^Nb_kx^k+\frac{1}{N}\sum_{k=0}^{\infty}k
b_kx^k\leqslant \sum_{k=0}^Nb_k +{\frac{x b'(x)}{N}}
\\
&\leqslant& \sum_{k=0}^Nb_k +b(x)\frac{x}{N}\frac{b'(x)}{b(x)}
\leqslant \sum_{k=0}^Nb_k +b(x)\frac{c x}{N(1-x)},
\end{eqnarray*}
here we have applied the inequality (\ref{logarithmic_derivative})
 satisfied by the logarithmic derivative of $b(x)$.
Inserting into the above inequality $x=e^{-2c/N}$ and estimating $\frac{e^{-2c/N}}{1-e^{-2c/N}}\leqslant \frac{N}{2c}$ we
obtain
\[
b(e^{-2c/N})
\leqslant
 \sum_{k=0}^Nb_k+\frac{1}{2}b(e^{-2c/N}),
\]
which leads to the inequality
\begin{equation}
\label{lower_bound_ineq}
\frac{1}{2}b(e^{-2c/N})\leqslant  \sum_{k=0}^Nb_k.
\end{equation}
 If $c\leqslant 1/2$, then $e^{-1/N}\leqslant e^{-2c/N}$ therefore
\[
\frac{1}{2}b(e^{-1/N})\leqslant \frac{1}{2}b(e^{-2c/N})\leqslant
\sum_{k=0}^Nb_k.
\]
This means that for $c\leqslant 1/2$ the theorem will be true with $K(c)=\frac12$.

Suppose now that $c>\frac{1}{2}$, then  $e^{-1/N}\geqslant e^{-2c/N}$. Let us show that the ratio $b(e^{-2c/N})/b(e^{-1/N})$ is bounded from bellow. Using the upper bound (\ref{logarithmic_derivative}) for the logarithmic derivative of $b(x)$ we get
\begin{equation*}
\begin{split}
\frac{b(e^{-2c/N})}{b(e^{-1/N})}&= \exp \left\{ \log b(e^{-2c/N})
-\log b(e^{-1/N})  \right\}= \exp \left\{
-\int_{e^{-2c/N}}^{e^{-1/N}}{\frac{b'(x)}{b(x)}}\,dx  \right\}
\\
&\geqslant \exp \left\{
-c\int_{e^{-2c/N}}^{e^{-1/N}}{\frac{dx}{1-x}} \right\} ={\left(
{\frac{1-e^{-1/N}}{1-e^{-2c/N}}} \right)}^c
\\
&\geqslant \frac{e^{-c/N}}{(2c)^c} \geqslant  \frac{e^{-1/2}}{(2c)^c},
\end{split}
\end{equation*}
since $N\geqslant 2c$. This estimate together with inequality (\ref{lower_bound_ineq}) proves
that the statement of the theorem is true with $K(c)=\frac{e^{-1/2}}{2(2c)^c}$, when $c>1/2$.

The theorem is proved.
\end{proof}
Throughout this section
 $p(z)$ will be defined as
\[
p(z)=\exp \left\{
 \sum_{k=1}^{\infty} \frac{d_k}{k}z^k \right\}=\sum_{j=0}^\infty
 p_jz^j.
\]
We will  assume that $d_k$ are bounded from above and
below by some fixed positive constants $0<d^-\leqslant d_k\leqslant
d^+$, and denote $\theta :=\min \{ 1,d^- \}$. We will also denote $\tilde
d_k=d_k-\theta$ and
$$
\tilde p(z)=\exp \left\{ \sum_{k=1}^{\infty}\frac{\tilde d_k}{k}z^k
\right\} =\sum_{n=0}^{\infty}\tilde p_nz^n.
$$
The relationship $\tilde d_k=d_k-\theta$ immediately leads to the
identity for the corresponding generating functions
$$
\tilde p(z)=\exp \left\{
 \sum_{j=1}^{\infty} \frac{d_j-\theta}{j}z^j \right\}={{p(z)}{(1-z)^{\theta}}}.
$$
In order to prove the inequality of our main Theorem \ref{fundthm1} we will need
some estimates for the asymptotic behavior of the coefficients of $p_j$ and $\tilde{p}_j$
of $p(z)$.

Differentiating $p(z)$ and $\tilde p(z)$ we conclude that these functions
satisfy differential equations
$$
zp'(z)=p(z)\sum_{k=1}^\infty d_kz^k \quad \hbox{and} \quad z\tilde
p'(z)=\tilde p(z)\sum_{k=1}^\infty \tilde d_kz^k,
$$
which lead to the recurrent relationships for the
coefficients $p_n$ and $\tilde{p}_n$ in  the Taylor expansions
of the corresponding functions
\begin{equation}
\label{p_n} p_n=\frac{1}{n}\sum_{k=1}^nd_kp_{n-k}\quad \hbox{and}
\quad \tilde
p_n=\frac{1}{n}\sum_{k=1}^n\tilde d_k{\tilde p_{n-k}},
\end{equation}
for $n\geqslant 1$. Taking the maximum of $d_k$ on the right hand
side of the above equations and  using  the inequality
(\ref{upper_bound_for_sum}) for partial sums of $p_j$ and
$\tilde{p}_j$ we obtain the inequalities
\begin{equation}
\label{p_n2} p_n\leqslant{d^+e} \frac{p(e^{-1/n})}{n}  \quad
\hbox{and} \quad \tilde p_n\leqslant {d^+e}\frac{\tilde
p(e^{-1/n})}{n}
\end{equation}
that provide an upper bound for the coefficients $p_n$ and $\tilde{p}_n$.

A similar lower bound for $p_n$ has been proven in \cite{manst_decomposable_2002},
stating that  there is a positive constant $K(d^+)$ such that
\begin{equation}
\label{plwbound}
p_n\geqslant {d^-K(d^+)}{\frac{p(e^{-1/n})}{n}}.
\end{equation}
An independent  proof of this estimate can be based on the
inequality of Theorem \ref{lowcoefth}. Indeed,  for $b(z)=p(z)$ we
see that the condition of  Theorem \ref{lowcoefth} is satisfied with
$c=d^+$ as $\frac{p'(x)}{p(x)}=\sum_{j=1}^\infty d_jx^{j-1}\leqslant
\frac{d^+}{1-x}$, for $0<x<1$. This gives us the lower bound for
partial sums of $p_j$, which together with recurrent relationships
(\ref{p_n}) satisfied by $p_n$ yields the proof of the lower bound
(\ref{plwbound}) for $p_n$.
\begin{lem}
\label{pvbound}
If $m\geqslant n\geqslant 1$, then
$$
{\left( {\frac{m}{n}}  \right) }^{d^-} e^{-{{d^-}/ {n}}}\leqslant
{\frac{p( e^{-{1/ {m}}})}{p( e^{-{1/ n} })}}\leqslant {\left(
\frac{m}{n} \right) }^{d^+} e^{{{d^+}/ {m}}},
$$
and
$$
{\left( {\frac{m}{n}}  \right) }^{\tilde d^-} e^{-{{\tilde d^-}/
{n}}} \leqslant {\frac{\tilde p( e^{-{1/ {m}}})}{\tilde p( e^{-{1/
n} })}}\leqslant {\left( \frac{m}{n}  \right) }^{\tilde d^+}
e^{{{\tilde d^+}/ {m}}},
$$
where $\tilde d_k^+=d_k^+ -\theta$, \ $\tilde d_k^-=d_k^- -\theta$
and $\tilde d^+=d^+-\theta$.
\end{lem}
\begin{proof}
We have
\begin{equation*}
\begin{split}
\frac{p( e^{-{1/ {m}}})}{p( e^{-{1/ n} })}&=\exp \left\{
\sum_{k=1}^{\infty}{\frac{d_k}{k}} ( e^{-{k/ {m}}} -e^{-{k/ n} })
\right\} \leqslant \exp \left\{ d^+\sum_{k=1}^{\infty} {\frac{
e^{-{k/ {m}}} -e^{-{k/ n} }}{k}} \right\}
\\
&=\exp \left\{ d^+ \log {\frac{1- e^{-{1/ n} }}{1- e^{-{1/ {m}}}}}
\right\}={\left( {\frac{1- e^{-{1/ n} }}{1- e^{-{1/ {m}}}}}
\right)}^{d^+} \leqslant {\left( {\frac{m}{n}}  \right) }^{d^+}
e^{{{d^+}/ {m}}}.
\end{split}
\end{equation*}
here we have used the inequalities $e^{-x}x\leqslant
1-e^{-x}\leqslant x$ for $x\geqslant 0$.

In the same way we obtain the lower bound estimate.

The proof of the second inequality is analogous.
\end{proof}
The next lemma proves that sequence $p_n$ varies "smoothly" in a certain sense.
\begin{lem}
\label{diffpnm}
If  $0 \leqslant s \leqslant n/2$, then
$$
|p_{n+s}-p_{n}| \ll  p_n \left( \frac{s}{n}\right)^{\theta},
$$
where $\theta =\min\{ d^-,1 \}$.
\end{lem}
\begin{proof}
Generating function $p(z)$ can be represented as a product
$p(z)={\frac{\tilde p(z)}{(1-z)^{\theta}}}$. This allow us
 to express the coefficients of $p(z)$ as a convolution of
 coefficients of the corresponding generating functions
$$
p_n=\sum_{k=0}^n\tilde p_k {\binom{n-k+\theta -1}{n-k}}.
$$
The idea behind the proof is to exploit  the fact that binomial
coefficients $\binom{n-k+\theta -1}{n-k}$ occurring in the above
expression  vary smoothly as follows from the classical estimate
\begin{equation}
\label{binomial_assympt}
 \binom{n+\theta -1}{n}={\frac{n^{\theta
-1}}{\Gamma (\theta)}}\left( 1+O\left( {\frac1n}\right) \right)
\end{equation}
(see e. g. \cite{flajolet_odlyzko_1990}). We can
represent $p_{n+s}-p_n$ as a difference of convolutions
\begin{eqnarray*}
p_{n+s}-p_n&=&\sum_{k=0}^n\tilde p_k \left( {\binom{n+s-k+\theta
-1}{n+s-k}} -{\binom{n-k+\theta -1}{n-k}} \right)
\\
&&\mbox{} +\sum_{n+s\geqslant k >n}\tilde p_k{\binom{n+s-k+\theta
-1}{n+s-k}}=:S_1+S_2.
\end{eqnarray*}
If $s=0$, then the estimate of the theorem is trivial, therefore we
assume that $s>0$. Applying here the upper bound (\ref{p_n2}) for
$p_j$ together with estimates for ratio \(\frac{\tilde
p(e^{-1/(n+s)})}{\tilde p(e^{-1/n})}\) provided by  Lemma
\ref{pvbound} we obtain
\begin{equation*}
\begin{split}
S_2&\leqslant \max_{n+s\geqslant k>n} \tilde p_k \sum_{l=0}^s
{\binom{l+\theta -1}{l}}   \leqslant  \tilde p(e^{-1/n})ed^+
 {\frac{\tilde p(e^{-1/(n+s)})}{\tilde
p(e^{-1/n})}} \frac{1}{n} \binom{s+\theta }{s}
\\
&\ll s^\theta \frac{\tilde p(e^{-1/n})}n=s^\theta {\frac{
p(e^{-1/n})(1-e^{-1/n})^\theta }n}\leqslant \left( \frac{s}{n}\right)^\theta  \frac{p(e^{-1/n})}{n}.
\end{split}
\end{equation*}
here while dealing with the binomial coefficient $\binom{s+\theta
}{s}$ we used (\ref{binomial_assympt}).

If $\theta =1$, then $S_1=0$, therefore estimating $S_1$ we may
assume that $\theta <1$. First we split the sum $S_1$ into two parts
and and notice that  the binomial coefficients
$\binom{l+\theta-1}{l}$ for $0<\theta<1$ are monotonously decreasing
when $l$ is increasing, which implies that $0<\binom{n-k+\theta -1}{n-k}
-\binom{n+s-k+\theta -1}{n+s-k} \leqslant {\binom{n-k+\theta
-1}{n-k}}$. Hence splitting sum $S_1$ into two parts we obtain
\[
\begin{split}
|S_1|&\leqslant \sum_{k=0}^{n-s}\tilde p_k{\binom{n-k+\theta
-1}{n-k}} \left|{\frac{\binom{n+s-k+\theta
-1}{n+s-k}}{\binom{n-k+\theta -1}{n-k}}}-1 \right|
\\
&\quad+ \sum_{n-s<k\leqslant n}\tilde p_k\binom{n-k+\theta -1}{n-k}
\end{split}
\]
 Once again applying Lemma \ref{pvbound}, the upper bound
 (\ref{p_n2}) for $\tilde{p}_j$ and the
 asymptotic (\ref{binomial_assympt}) for binomial coefficients, we
 obtain
\[
\begin{split}
S_1 &\ll \sum_{k=0}^{n-s}\tilde p_k {\binom{n-k+\theta -1}{n-k}}
{\frac{s}{n-k}} + \frac{\tilde p(e^{-1/n})}{n}\sum_{0\leqslant
 j<s}\binom{j+\theta -1}{j}
\\
&\ll  {\frac{s}{n}}n^{\theta -1}\sum_{k\leqslant n/2}\tilde p_k
+\sum_{n/2<k\leqslant n-s}\tilde p_k s(n-k)^{\theta-2} +
\frac{p(e^{-1/n})}{n}\left( \frac{s}{n} \right)^\theta
\\
&\ll   \frac{s}{n}n^{\theta-1}\tilde p(e^{-1/n})+
 {\frac{\tilde p(e^{-1/n})}{n}}s\sum_{l\geqslant s}l^{\theta -2}
+\frac{p(e^{-1/n})}{n}\left( \frac{s}{n} \right)^\theta
\\
&\ll  \frac{p(e^{-1/n})}{n}\frac{s}{n}+ \frac{p(e^{-1/n})}{n}\left(
\frac{s}{n} \right)^\theta \ll \frac{p(e^{-1/n})}{n}\left( \frac{s}{n} \right)^\theta.
\end{split}
\]
The lemma is proved.
\end{proof}
For $0\leqslant x \leqslant 1$ we denote
$$
G_x(z)=\frac{p(z)}{p(zx)}=\sum_{k=0}^\infty g_{k,x}z^k \quad
\hbox{and} \quad \tilde G_x(z)=\frac{\tilde p(z)}{\tilde
p(zx)}=\sum_{k=0}^\infty \tilde g_{k,x}z^k,
$$
and
$$
C_x(z)= \left( \frac{1-zx}{1-z}
\right)^\theta=\sum_{k=0}^\infty c_{k,x}z^k.
$$
Since $\tilde p(z)={{p(z)}{(1-z)^{\theta}}}$, we have
$$
G_x(z)=\tilde G_x(z) \left( \frac{1-zx}{1-z} \right)^\theta .
$$
Differentiating functions $C_x(z)$ and $G_x(z)$ with respect to $z$ we conclude that
they satisfy differential equations
$$
zC_x'(z)=C_x(z)\theta \sum_{k=1}^\infty z^k(1-x^k)\quad \hbox{and}
\quad zG_x'(z)=G_x(z)\sum_{k=1}^\infty d_k z^k(1-x^k).
$$
which are equivalent to the recurrent relationships
\begin{equation}
\label{recurrences_for_c_n_and_g_n}
c_{n,x}=\frac{\theta }n \sum_{k=1}^n c_{n-k,x}(1-x^k)\quad
\hbox{and} \quad g_{n,x}=\frac1n\sum_{k=1}^n
g_{n-k,x}d_k(1-x^k),
\end{equation}
satisfied by the coefficients $c_{n,x}$, $g_{n,x}$ in the  Taylor expansions of the corresponding functions,
for  all $n\geqslant 1$ with initial conditions $c_{0,x}=g_{0,x}=1$. Since the coefficients of
these linear recurrences together with the  initial conditions are non-negative, thus it follows
by induction that the solutions $c_{n,x},g_{n,x}$ of the above recurrences are also non-negative
$c_{n,x},g_{n,x}\geqslant 0$. Therefore we can apply our inequality (\ref{upper_bound_for_sum})
for the partial sums of coefficients of generating functions with non-negative coefficients
to obtain the upper bounds
$$
\sum_{m=0}^nc_{m,x}\leqslant e C_x(e^{-1/n}) \quad \hbox{and}\quad
\sum_{m=0}^ng_{m,x}\leqslant e G_x(e^{-1/n}).
$$
Applying these upper bounds for partial sums of  $c_{n,x},g_{n,x}$
to bound the right hand side of the recurrences
(\ref{recurrences_for_c_n_and_g_n}) satisfied by these coefficients
we obtain the inequalities
\begin{equation}
\label{c_n_x_inequality}
 c_{n,x}\leqslant
\frac{e\theta C_x(e^{-1/n}) }{n} \quad \hbox{and}\quad
g_{n,x}\leqslant \frac{ed^+ G_x(e^{-1/n}) }{n}.
\end{equation}
The same considerations applied to generating function
$\tilde{G}_x(z)$ lead to inequality
\begin{equation}
\label{g_tilde_n,x_inequality} \tilde{g}_{n,x}\leqslant
\frac{e(d^+-\theta) \tilde{G}_x(e^{-1/n}) }{n}.
\end{equation}
\begin{lem}
\label{diffcmn}
Suppose $0<x<1$ and $s\leqslant
m/2$, then we have
$$
|c_{m,x}-c_{m-s,x}|\ll sm^{\theta-2}(1-x)^\theta +\frac{s}{m^2},
$$
for $m\geqslant 1$.
\end{lem}
\begin{proof}
We will apply the same standard technique of contour integration
that was used by Flajolet and Odlyzko  \cite{flajolet_odlyzko_1990}
to analyze generating functions with singularities of type
$1/(1-z)^\alpha$. The first step in our proof is to represent
$c_{m,x}$ as a Cauchy integral of function $C_x(z)$ over contour $L$
that consists of four parts $ L=L_1\cup L_2\cup L_3 \cup L_4 $, two
arcs $L_1$, $L_2$ with radiuses $2$, $1/m$ correspondingly and two
segments $L_3$ and $L_4$ connecting the ends of these arcs as shown
bellow on the picture.

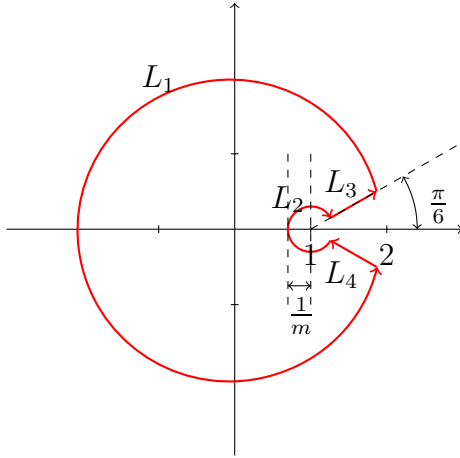
\begin{tikzpicture}
  \draw[->] (-3,0)--(3,0);
  \draw[->] (0,-3)--(0,3);
  \draw[->,thick,red] (1.866025404,0.5) arc (15:346:2);
  \draw[<-,thick,red] (1.259807621,0.15) arc (30:330:0.3);
  \draw[->,thick,red] (1.259807621,0.15)--(1.866025404,0.5);
  \draw[<-,thick,red] (1.259807621,-0.15)--(1.866025404,-0.5);
  \draw (1,0.05) -- (1,-0.05) node[anchor=north] {$1$};
  \draw (2,0.05) -- (2,-0.05) node[anchor=north] {$2$};
  \draw (-0.05,1) -- (0.05,1);
  \draw (-0.05,-1) -- (0.05,-1);
  \draw (-1,0.05) -- (-1,-0.05);
  \draw[dashed,] (1,0)--(2.94850, 1.125);
  \draw[<->] (2.4,0) arc (0:30:1.4);
  \draw[dashed] (0.7,1) -- (0.7,-1);
  \draw[dashed] (1,1) -- (1,-1);
  \draw[,<->] (0.7,-0.75) -- (1,-0.75);
  \node[anchor=north] at (0.875,-0.75) {$\frac{1}{m}$};
  \node[anchor=west] at (2.4,0.3){$\frac\pi6$};
  \node at (-1,2){$L_1$};
  \node at (0.7,0.4){$L_2$};
  \node at (1.4,0.6){$L_3$};
  \node at (1.4,-0.6){$L_4$};
\end{tikzpicture}

Thus we can replace the difference of coefficients $c_{m,x}$ and $c_{m-1,x}$ by
the difference of the corresponding Cauchy integrals
\begin{equation*}
\begin{split}
|c_{m,x}-c_{m-1,x}|&= \left| \frac{1}{2\pi i}\int_{L}
C_x(z) {\frac{(1-z)}{z^{m+1}}}\,dz \right| \leqslant
 {\frac1{2\pi}}\int_{L}
{\frac{|1-xz|^\theta|1-z|^{1-\theta}}{|z|^{m+1}}}\,|dz|
\\
&\leqslant
 \frac{1}{2\pi}\int_{L}
{\frac{(|1-z|+|z||1-x|)^\theta|1-z|^{1-\theta}}{|z|^{m+1}}}\,|dz|
\\
&\ll
 \int_{L}
{\frac{|1-z|+|1-z|^{1-\theta}|1-x|^\theta}{|z|^{m+1}}}\,|dz|.
\end{split}
\end{equation*}
Let us now estimate the above integral over the four separate parts of our contour
\begin{equation*}
\begin{split}
|c_{m,x}-c_{m-1,x}|&\ll \frac{1}{2^m}
+\int_{1+{\frac1m}}^2 {\frac{(y-1)+(y-1)^{1-\theta}|1-x|^\theta}{y^{m+1}}}\,dy
\\
&\quad+\int_{|z-1|={\frac1m}} {\frac{|1-z|+|1-z|^{1-\theta}|1-x|^\theta}{|z|^{m+1}}}\,|dz|
\\
&\ll\frac{(1-x)^\theta}{m}\int_{m\log \left( 1+{\frac1m}\right)
}^{m\log 2} \frac{(e^{u/m}-1)^{1-\theta}}{e^u} \,du
+\frac{1}{m^2}+m^{\theta -2}(1-x)^\theta
\\
&\ll(1-x)^\theta m^{\theta -2}\int_{m\log \left( 1+{\frac1m}\right)
}^{m\log 2} {u^{\theta}}{e^{-u}} \,du +\frac{1}{m^2}+m^{\theta
-2}(1-x)^\theta
\\
&\ll
 m^{\theta-2}(1-x)^\theta +\frac1{m^2}.
 \end{split}
\end{equation*}
From the above estimate of the difference $|c_{m,x}-c_{m-1,x}|$ by the standard use of telescoping sums
we obtain the estimate
\begin{equation*}
\begin{split}
|c_{m,x}-c_{m-s,x}|&\leqslant |c_{m,x}-c_{m-1,x}|+|c_{m-1,x}-c_{m-2,x}|+\cdots
+|c_{m-s+1,x}-c_{m-s,x}|
\\
&\ll sm^{\theta-2}(1-x)^\theta +{\frac{s}{m^2}},
\end{split}
\end{equation*}
for $s\leqslant m/2$.

The lemma is proved.
\end{proof}

\begin{lem}
\label{diffgnm}
For $0\leqslant x \leqslant e^{-1/n}$ and
$k\leqslant n/8$, we have
$$
g_{n,x} -g_{n-k,x} \ll {\frac{p(e^{-1/n})}{np(x)}}\left( \left(
{\frac{k}{n}} \right)^\theta +{\frac1{(n(1-x))^\theta}} \right).
$$
\end{lem}
\begin{proof}
Since $G_x(z)=\tilde G_x(z) C_x(z)$ this allows us to express the
difference $g_{n,x}$ as a convolution of coefficients of $\tilde
G_x(z)$ and $C_x(z)$.Thus we can to express the difference of
coefficients $g_{n,x}$ as
\begin{equation*}
\begin{split}
g_{n,x}-g_{n-k,x}&=
\sum_{s=0}^n \tilde g_{s,x}c_{n-s,x}
-\sum_{s=0}^{n-k} \tilde g_{s,x}c_{n-k-s,x}
=\sum_{s=0}^{n-2k}\tilde g_{s,x}(c_{n-s,x}-c_{n-k-s,x})
\\
&\quad+\sum_{n-2k<s \leqslant n} \tilde g_{s,x}c_{n-s,x}
-\sum_{n-2k<s\leqslant n-k} \tilde g_{s,x}
c_{n-k-s,x}=:S_1+S_2+S_3.
\end{split}
\end{equation*}
If $\theta=1$ then all $c_{j,x}$ with $j\geqslant 1$ are equal
$c_{j,x}=1-x$. Which means that $S_1=0$ if $\theta=1$. Therefore,
while estimating $S_1$, we may assume that $\theta<1$.
 Applying the estimate of the difference $c_{j+s,x}-c_{j,x}$ of Lemma \ref{diffcmn}
 together with the inequality
 (\ref{upper_bound_for_sum}) for partial sums of coefficients of
generating function with positive coefficients and making use of the
upper bound (\ref{g_tilde_n,x_inequality}) for $\tilde{g}_{j,x}$
  we obtain
\[
\begin{split}
S_1&\ll \sum_{0\leqslant s \leqslant n-2k} \tilde g_{s,x}\left(
 k(n-s)^{\theta-2}(1-x)^\theta +\frac{k}{(n-s)^2}
     \right)
\\
&\ll  kn^{\theta-2}(1-x)^\theta \sum_{0\leqslant s \leqslant n/2}
\tilde g_{s,x}
\\
&\quad+\sum_{n/2 \leqslant s \leqslant n-2k} \tilde g_{s,x}\left(
 k(n-s)^{\theta-2}(1-x)^\theta +\frac{k}{(n-s)^2}
     \right)
\\
&\ll  kn^{\theta-2}(1-x)^\theta \tilde G_x(e^{-1/n})
\\
&\quad+k{\frac{\tilde G_x(e^{-1/n})}n} \sum_{n/2 \leqslant s
\leqslant n-2k} \left(
 (n-s)^{\theta-2}(1-x)^\theta +\frac{1}{(n-s)^2}
     \right)
\\
&\ll  kn^{\theta-2}(1-x)^\theta  G_x(e^{-1/n}) \left(
\frac{1-e^{-1/n}}{1-xe^{-1/n}} \right)^\theta
\\
&\quad+k \frac{G_x(e^{-1/n})}n \left( {\frac{1-e^{-1/n}}{1-xe^{-1/n}}} \right)^\theta \left(
 k^{\theta-1}(1-x)^\theta +{\frac1k}
     \right).
\end{split}
\]
Since   $1-xe^{-1/n}\geqslant 1-x$, we have
$$
S_1\ll {\frac{k}{n^2}}{\frac{p(e^{-1/n})}{p(xe^{-1/n})}} +{\frac1n}{\frac{p(e^{-1/n})}{p(xe^{-1/n})}}
\left( \left( \frac{k}{n}
\right)^\theta +\frac{1}{(n(1-x))^\theta} \right).
$$
In a similar way we obtain
\begin{equation*}
\begin{split}
S_2+S_3&\ll {\frac1n} \tilde G_x(e^{-1/n})\sum_{0\leqslant l \leqslant 2k}c_{x,l}
\\
&\ll {\frac1n}{\frac{p(e^{-1/n})}{p(xe^{-1/n})}}
 \left(
{\frac{1-e^{-1/n}}{1-xe^{-1/n}}}
\right)^\theta    \left(
{\frac{1-xe^{-1/2k}}{1-e^{-1/2k}}}
\right)^\theta
\\
&\ll {\frac1n}{\frac{p(e^{-1/n})}{p(xe^{-1/n})}}
 \left(
{\frac{1-x +x(1-e^{-1/2k})}{1-x}}
\right)^\theta   \left( {\frac{k}{n}} \right)^\theta
\\
&\ll {\frac1n} {\frac{p(e^{-1/n})}{p(xe^{-1/n})}}
 \left(
1+{\frac1{k(1-x)}}
\right)^\theta  \left( {\frac{k}{n}} \right)^\theta
\\
&\ll {\frac1n} {\frac{p(e^{-1/n})}{p(xe^{-1/n})}}
\left( \left( \frac{k}{n}
\right)^\theta +{\frac1{(n(1-x))^\theta}} \right).
\end{split}
\end{equation*}
Since $p(xe^{-1/n})\gg p(x)$ if $0\leqslant x \leqslant e^{-1/n}$,
 the proof of the lemma follows.
\end{proof}

\begin{lem}
\label{intbounds}
Suppose $u(x)=\exp \left\{
\sum_{k=1}^{\infty}\frac{u_k}{k}x^k
  \right\}$ and $0\leqslant u_k \leqslant A$, then the following estimates hold:
\begin{enumerate}
\item[1)]
$$
\quad\int_0^1{\frac{x^{j-1}}{u(x)}}\,dx \ll {\frac1{ju(e^{-1/j})}}
\quad \hbox{if} \quad j\geqslant 1;
$$
\item[2)]
$$
\quad\int_0^{e^{-1/n}}{\frac{x^{j-1}}{u(x)}}\,dx \ll
\frac{e^{-j/n}}{ju(e^{-1/n})} \quad \hbox{if} \quad j\geqslant n\geqslant 1.
$$
\end{enumerate}
\end{lem}
\begin{proof} 1) For $j\geqslant 1$ we can split the integral into
 two parts and taking into account that $u(x)$ is increasing obtain
\begin{equation*}
\begin{split}
\int_0^1{\frac{x^{j-1}}{u(x)}}\,dx
&=\int_0^{e^{-1/j}}{\frac{x^{j-1}}{u(x)}}\,dx
+\int_{e^{-1/j}}^1\frac{x^{j-1}}{u(x)}\,dx
\\
&\leqslant {\frac1{u(e^{-1/j})}}\int_0^{e^{-1/j}}{\frac{u(e^{-1/j})}{u(x)}}
x^{j-1}\,dx +{\frac{1-e^{-1/j}}{ju(e^{-1/j})}}
\\
&\leqslant {\frac1{u(e^{-1/j})}}\int_0^{e^{-1/j}} x^{j-1}\exp \left\{
\sum_{k=1}^{\infty}{\frac{u_k}{k}}(e^{-k/j}-x^k)
  \right\}\,dx+ {\frac1{ju(e^{-1/j})}}.
\end{split}
\end{equation*}
Applying now the upper bound $u_k\leqslant A$ to estimate the quantity under the integration sign
we obtain
\[
\begin{split}
\int_0^1{\frac{x^{j-1}}{u(x)}}\,dx&\leqslant
{\frac1{u(e^{-1/j})}}\int_0^{e^{-1/j}} x^{j-1} \left(
{\frac{1-x}{1-e^{-1/j}}} \right)^A   \,dx+ {\frac1{ju(e^{-1/j})}}
\\
&\leqslant {\frac{e^{A/j}j^A}{u(e^{-1/j})}}\int_0^{e^{-1/j}}
x^{j-1}(1-x)^A\,   dx+ {\frac1{ju(e^{-1/j})}}
\\
&= {\frac{e^{A/j}j^A}{u(e^{-1/j})}}\int_{1/j}^{\infty }
(1-e^{-y})^A e^{-jy}\,   dy + {\frac1{ju(e^{-1/j})}}
\\
&\leqslant {\frac{e^{A/j}j^A}{u(e^{-1/j})}}\int_0^{\infty} y^A
e^{-jy}\, dy + {\frac1{ju(e^{-1/j})}} = \frac{e^{A/j} \Gamma
(A+1)+1}{ju(e^{-1/j})},
\end{split}
\]
here we have used the inequalities $e^{-y}y\leqslant
1-e^{-y}\leqslant y$, for $y\geqslant 0$.

2) Suppose now that $j\geqslant n$. Applying the same considerations
that we used to estimate the integral over region $0\leqslant x
\leqslant e^{-1/n}$ in the previous estimate, we obtain
\begin{equation*}
\begin{split}
\int_0^{e^{-1/n}}{\frac{x^{j-1}}{u(x)}}\,dx
&\leqslant {\frac1{u(e^{-1/n})}}\int_0^{e^{-1/n}} x^{j-1} \left( {\frac{1-x}{1-e^{-1/n}}} \right)^A \,  dx
\\
&\leqslant
 {\frac{e^{A/n}n^A}{u(e^{-1/n})}}\int_0^{e^{-1/n}}
x^{j-1}(1-x)^A  \, dx
\\
&={\frac{e^{A/n}n^A}{u(e^{-1/n})}}\int_{1/n}^{\infty }
(1-e^{-y})^Ae^{-jy}  \, dy
\\
&\leqslant {\frac{e^{A/n}n^A}{u(e^{-1/n})}}\int_{1/n}^{\infty } y^Ae^{-jy}  \, dy
\\
&={\frac{e^{A/n}n^A}{u(e^{-1/n})}}{\frac1{j^{A+1}}}\int_{j/n}^{\infty } y^Ae^{-y} \,  dy \ll {\frac{e^{-j/n}}{ju(e^{-1/n})}}.
\end{split}
\end{equation*}
The last inequality follows from the fact that $\int_{T}^{\infty }
y^Ae^{-y} \, dy \ll T^Ae^{-T}$, as $T\to \infty$.

The lemma is proved.
\end{proof}

\begin{lem}
\label{boundofg}
\[
\int_0^{e^{-1/n}} \left| g_{x,n} -{\frac{p_n}{p(x)}}
\right|x^{j-1}\, dx \ll {\frac{p(e^{-1/n})}{n} } \left( {\frac{j^{\theta
-1}}{n^\theta}} \right) {\frac1{p(e^{-1/j})}},
\]
when $1\leqslant j\leqslant n$.
\end{lem}

\begin{proof} We can represent function $p(z)$ as a product
$p(z)=p(xz){\frac{p(z)}{p(xz)}}=p(xz)G_x(z)$ which is equivalent to
representation of the coefficient $p_n$ of $p(z)$ as a convolution
$p_n=\sum_{k=0}^np_kx^kg_{n-k,x}$. Applying this identity we obtain
\begin{equation*}
\begin{split}
\left| g_{x,n} -{\frac{p_n}{p(x)}}  \right| &=
 {\frac1{p(x)}}\left|p(x)g_{n,x} - \sum_{k=0}^np_kx^kg_{n-k,x} \right|
\\
&\leqslant {\frac1{p(x)}} \sum _{k\leqslant n/8}p_{k}x^{k}\left|
g_{n,x}- g_{n-k,x} \right|  +{\frac{g_{n,x}}{p(x)}}\sum_{k>n/8}p_{k}x^{k}
\\
&\quad+{\frac1{p(x)}}\sum_{n/8< k
\leqslant n}p_{k}x^{k}g_{n-k,x},
\end{split}
\end{equation*}
Suppose $0\leqslant x\leqslant e^{-1/n}$. Applying here Lemma
\ref{diffgnm} and the upper bound (\ref{p_n2}) for $p_n$ we have
\begin{equation*}
\begin{split}
\left| g_{x,n} -{\frac{p_n}{p(x)}}  \right|&\ll
{\frac{p(e^{-1/n})}{np(x)^{2}}}\sum_{k\leqslant
n/8}p_{k}x^{k}\left( \left( \frac{k}{n} \right)^\theta + {\frac1{(n(1-x))^\theta }}
 \right)
\\
&\quad+{\frac{G_x(e^{-1/n})}{np(x)}}\sum_{k> n/8}p_{k}x^{k}+
{\frac{p(e^{-1/n})}{n}}{\frac{x^{n/8}}{p(x)}}G_x(e^{-1/n})
\\
&=:S_{1}(x) +S_{2}(x) +S_{3}(x).
\end{split}
\end{equation*}
Applying Lemma \ref{intbounds} with $u(x)=p(x)^2$ and
 $u(x)=p(x)^2(1-x)^\theta$, for $1\leqslant j \leqslant n$ we have
\begin{multline*}
\int_{0}^{e^{-1/n}} S_{1}(x)x^{j-1}\,dx
\\
=\frac{p(e^{-1/n})}{n} \sum_{k\leqslant n/8}p_k\left( \left(
\frac{k}{n}\right)^\theta \int_{0}^{e^{-1/n}}
{\frac{x^{k+j-1}}{p(x)^2}}\,dx +{\frac1{n^\theta}}
\int_{0}^{e^{-1/n}} {\frac{x^{k+j-1}dx}{{p(x)^2}(1-x)^\theta }}
\right)
\\
\shoveleft{\ll {\frac{p(e^{-1/n})}{n}} \sum_{k\leqslant n/8}p_k\left( \left(
\frac{k}{n}\right)^\theta  {\frac1{(k+j)p(e^{-1/(k+j)})^2}}\right.}
\\
\shoveright{ \left.+{\frac1{n^\theta}}
\frac1{(k+j){p(e^{-1/(k+j)})^2}(1-e^{-1/(k+j)})^\theta }
\right)}
\\
\shoveleft{\ll {\frac{p(e^{-1/n})}{n}} \sum_{k\leqslant n/8}
{\frac{p_k}{(k+j)p(e^{-1/(k+j)})^2}}\left( \left( \frac{k}{n}\right)^\theta
 +\frac{(k+j)^\theta}{n^\theta}
\right)}
\\
\shoveleft{\ll {\frac{p(e^{-1/n})}{n}} \left(
 \frac{1}{j}\left( \frac{j}{n}\right)^\theta
{\frac1{p(e^{-1/j})^2}} \sum_{k\leqslant j} {p_k}
 +\sum_{k>j} {\frac{p_k}{k}} \left( \frac{k}{n}\right)^\theta
 {\frac1{p(e^{-1/k})^2}} \right).\hfill}
\end{multline*}
Applying here inequality $\sum_{k\leqslant j} p_k\leqslant
ep(e^{-1/j})$ in the first sum and the estimate (\ref{p_n2}) in the second one
we have
\[
\begin{split}
&\int_{0}^{e^{-1/n}} S_{1}(x)x^{j-1}\,dx
\\
&\ll \frac{p(e^{-1/n})}{n} \left(
 \frac{1}{j}\left( \frac{j}{n}\right)^\theta
{\frac1{p(e^{-1/j})}}
 +{\frac1{p(e^{-1/j})}}\sum_{k>j} {\frac1{k^2}}
 \left( \frac{k}{n}\right)^\theta \frac{p(e^{-1/j})}{p(e^{-1/k})} \right).
\end{split}
\]
 From the estimate
${\frac{p(e^{-1/j})}{p(e^{-1/k})}}\leqslant \left( \frac{j}{k}
\right)^{d^-}e^{d^-/j}$ of Lemma \ref{pvbound} we obtain
\[
\begin{split}
&\int_{0}^{e^{-1/n}} S_{1}(x)x^{j-1}\,dx
\\
&\ll {\frac{p(e^{-1/n})}{n}} \left(
 \frac{1}{j}\left( \frac{j}{n}\right)^\theta
{\frac1{p(e^{-1/j})}}
 +{\frac1{p(e^{-1/j})}}\sum_{k>j} {\frac1{k^2}}
 \left( \frac{k}{n}\right)^\theta \left( \frac{j}{k} \right)^{d^-}  \right)
\\
&\ll {\frac{p(e^{-1/n})}{n} } \left( {\frac{j^{\theta -1}}{n^\theta}}
\right) {\frac1{p(e^{-1/j})} }.
\end{split}
\]

Let us now estimate $S_2(x)$
$$
S_2(x)\ll \frac{p(e^{-1/n})}{np(x)^2}  \sum_{k>n/8}p_kx^k\leqslant
{\frac{p(e^{-1/n})x^{n/16}}{np(x)^2} }p(\sqrt{x}) \ll
{\frac{p(e^{-1/n})x^{n/16}}{np(x)} },
$$
since $p(y)\ll p(y^2)$ uniformly for $0\leqslant y <1$. Hence,
applying Lemma \ref{intbounds} we have
$$
\int_{0}^{e^{-1/n}}S_2(x)x^{j-1}\,dx \ll {\frac{p(e^{-1/n})}{n} }
\int_{0}^{e^{-1/n}}\frac{x^{n/16}}{p(x)}\,dx\ll \frac1{n^2}.
$$
In a similar  way we obtain the estimate
$$
\int_{0}^{e^{-1/n}}S_3(x)x^{j-1}\,dx \ll {\frac{p(e^{-1/n})^2}{n} }
\int_{0}^{e^{-1/n}}{\frac{x^{n/8}}{p(x)^2}} \,dx\ll {\frac1{n^2}}.
$$
Collecting the obtained estimates and noticing that
$$
\frac{p(e^{-1/n})}{ n } \left( {\frac{j^{\theta -1}}{n^\theta}}
\right) {\frac1{p(e^{-1/j})}}\gg \frac{1}{nj} \geqslant \frac1{n^2},
$$
for $1\leqslant j \leqslant n$, we obtain the proof of the lemma.
\end{proof}
Let us define
$$
V_j(z)=p(z)\int_{0}^{1}\frac{x^{j-1}}{p(zx)}\,dx =\sum_{m=0}^{\infty
}v_{m,j}z^m.
$$
Denoting
$$
q(z)=\frac1{p(z)}=\sum_{m=0}^{\infty}q_mz^m
$$
we can represent $V_j(z)$ as a product of two Taylor series
$$
V_j(z)=
 \sum_{m=0}^{\infty }v_{m,j}z^m=p(z)\int_{0}^{1}{{x^{j-1}}{q(zx)}}\,dx
=\sum_{m=0}^{\infty}p_mz^m \sum_{s=0}^{\infty}\frac{q_s}{s+j}z^s,
$$
which leads to representations of the Taylor coefficients of
$V_j(z)$ as a convolution of the coefficients of the appropriate
series
$$
v_{m,j}=\sum_{s=0}^m{\frac{p_{m-s}q_s}{s+j}}.
$$
On the other hand,
$$
V_j(z)=\int_{0}^{1}x^{j-1}G_x(z)\,dx=\sum_{m=0}^{\infty}z^m\int_0^1g_{m,x}x^{j-1}\,dx=\sum_{m=0}^{\infty
}v_{m,j}z^m,
$$
which means that
\begin{equation}
\label{v_positive}
v_{m,j}=\int_0^1g_{m,x}x^{j-1}\,dx\geqslant 0 .
\end{equation}
\begin{lem}
\label{f_mj}
We have
$$
v_{m,j}\ll \frac{1}{j^2}\quad \hbox{when}\quad 1\leqslant m\leqslant
j
$$
and  $v_{0,j}=\frac1j$.
\end{lem}

\begin{proof}
Differentiating $V_j(z)$ we can easily check that this function
satisfies differential equation
$$
zV_j'(z)=V_j(z)\sum_{k=1}^{\infty}d_kz^k +1-jV_j(z).
$$
Putting here $z=0$, we obtain $V_j(0)=v_{0,j}={1/j}$.

Suppose that $m\geqslant 1$. The above differential equation for
$V_j(z)$ is equivalent to the recurrence relation between the Taylor
coefficients $v_{m,j}$, applying which  we obtain
\begin{equation*}
\begin{split}
v_{m,j}&=\frac{1}{m+j}\sum_{k=1}^m d_kv_{m-k,j}\leqslant
\frac{d^+eV_j(e^{-1/m})}{m+j}=\frac{d^+e}{m+j}p(e^{-1/m})\int_0^1
\frac{x^{j-1}dx}{p(xe^{-1/m})}
\\
&=\frac{d^+e}{m+j}p(e^{-1/m})e^{j/m}\int_0^{e^{-1/m}}
\frac{x^{j-1}dx}{p(x)}\ll {\frac{1}{j^2}},
\end{split}
\end{equation*}
for  $j\geqslant m$. Here we have used the fact that $v_{m,j}$ are
non-negative (\ref{v_positive}) and applied Lemma \ref{intbounds}.

\end{proof}

\begin{proof}[ Proof of theorem \ref{fundthm1}.]
The first step of the proof
is to express the coefficients $a_k$ in terms of a linear
combination of quantities $S(g;j)$. By the definition of quantity
$S(g;j)$ we see that its generating function can be expressed as
$$
\sum_{j=1}^{\infty}S(g;j)z^{j}=zg'(z)p(z).
$$
Dividing both sides of the above identity by $p(z)$ we obtain an
expression of derivative $g'(z)$ as a product of two functions
\[
q(z)\sum_{j=1}^{\infty}S(g;j)z^{j}=zg'(z),
\]
where
$$
q(z)=\frac{1}{p(z)}=\sum_{j=0}^{\infty}q_jz^j,
$$
which is equivalent to the identity expressing the Taylor
coefficient $ma_m$ of the derivative $g'(z)$ as a convolution
$$
ka_k=\sum_{j=1}^{k}S(g;j)q_{k-j}.
$$
We can now use this identity to replace coefficients $a_k$ by a sum
$(1/k)\sum_{j=1}^{k}S(g;j)q_{k-j}$ in the expression appearing on
the left hand side of the inequality (\ref{fund_ineq}) of our theorem
\begin{equation*}
\begin{split}
\lefteqn{\sum_{k=0}^na_kp_{n-k}-p_n g(e^{-1/n})}
\\
&=\sum_{k=1}^{n}p_{n-k}{\frac1k}\sum_{j=1}^{k}S(g;j)q_{k-j}
-p_n\sum_{k=1}^{\infty} {\frac{e^{-k/n}}{k}
}\sum_{j=1}^{k}S(g;j)q_{k-j}
\\
&= \sum_{j=1}^n S(g;j)\sum_{k=j}^n {\frac{p_{n-k}q_{k-j}}{k}}
-p_n\sum_{j=1}^{\infty}S(g;j)\sum_{k=j}^{\infty}{\frac{e^{-k/n}}{k}}q_{k-j}
\\
&=\frac{S(g;n)}{n}+\sum_{j=1}^{n-1}S(g;j)\sum_{s=0}^{n-j}\frac{p_{n-j-s}q_s}{s+j}-p_n
\sum_{j=1}^{\infty}S(g;j)\sum_{k=j}^{\infty}\frac{e^{-k/n}}{k}q_{k-j}.
\end{split}
\end{equation*}
recalling that
\[
v_{m,j}=\sum_{s=0}^m\frac{p_{m-s}q_s}{s+j}=\int_0^1g_{m,x}x^{j-1}\,dx
\]
obtain an identity
\begin{multline*}
\sum_{k=0}^na_kp_{n-k}-p_n g(e^{-1/n}) -\frac{S(g;n)}{n}
\\
=\sum_{j=1}^{n-1}S(g;j)\left( v_{n-j,j} -p_n \int_{0}^{e^{-1/n}}
\frac{x^{j-1}}{p(x)}\,dx \right)
\\
+p_n \sum_{j=n}^{\infty}S(g;j)\int_{0}^{e^{-1/n}} \frac{x^{j-1}}{p(x)}\,dx.
\end{multline*}
For brevity, let us denote by $R_n$ the quantity on the left hand side of the above identity
$$
R_n=\sum_{k=0}^na_kp_{n-k}-p_n g(e^{-1/n})-\frac{S(g;n)}{n}.
$$
Then we have
\begin{equation}
\label{|R_n|}
\begin{split}
|R_n|&\leqslant \sum_{1\leqslant j \leqslant n/2}|S(g;j)|\left|
\int_{0}^{1}{x^{j-1}}g_{n-j,x} \,dx -p_{n-j} \int_{0}^{e^{-1/n}}
{\frac{x^{j-1}}{p(x)}}dx \right|
\\
&\quad+\sum_{1\leqslant j \leqslant n/2}|S(g;j)| |p_n -p_{n-j}|
\int_{0}^{e^{-1/n}} {\frac{x^{j-1}}{p(x)}}\,dx
\\
&\quad + \sum_{n/2\leqslant j \leqslant n-1}|S(g;j)| v_{n-j,j} +p_n
\sum_{j\geqslant n/2}|S(g;j)|\int_{0}^{e^{-1/n}} \frac{x^{j-1}}{p(x)}\,dx.
\end{split}
\end{equation}
Since by (\ref{c_n_x_inequality}) we have $g_{m,x}  \leqslant
ed^+{\frac{G_j(e^{-1/m})}{m}}= {\frac{ed^+}{m}}
{\frac{p(e^{-1/m})}{p(e^{-1/m}x)}}$, therefore
\[
\begin{split}
\int_{e^{-1/n}}^{1}{x^{j-1}}g_{n-j,x} \,dx&\leqslant
{\frac{ed^+p(e^{-1/(n-j)})}{n-j}}\int_{e^{-1/n}}^{1}
{\frac{dx}{p(xe^{-1/(n-j)})}}
\\
&\ll
{\frac{p(e^{-1/n})}{n^2}}{\frac{1}{p(e^{-1/n})}}={\frac{1}{n^2}},
\end{split}
\]
when $j\leqslant n/2$.

This estimate allows us to evaluate the part
of the integral  of ${x^{j-1}}g_{n-j,x}$ over interval
$e^{-1/n}\leqslant x\leqslant 1$ in the first sum on the right hand
side of inequality (\ref{|R_n|}), making the application of Lemma
\ref{boundofg} possible to estimate the difference of integrals in
the first sum. The second sum in the inequality (\ref{|R_n|}) can be
evaluated using upper bound for difference $|p_n-p_{n-j}|$ of Lemma
\ref{diffpnm} and the estimate of integral $\int_{0}^{e^{-1/n}}
{\frac{x^{j-1}}{p(x)}}\,dx\ll \frac{1}{jp(e^{-1/j})}$ provided by
Lemma  \ref{intbounds}. Finally, applying the estimate  $v_{n-j,j}\ll
1/j^2$ of Lemma
 \ref{f_mj}
 that is valid in the region $j\geqslant n/2$ to evaluate the third sum and the upper  bound
 $\int_{0}^{e^{-1/n}} {\frac{x^{j-1}}{p(x)}}\,dx\ll\frac{e^{-j/n}}{jp(e^{-1/n})}$
 of Lemma \ref{intbounds} to evaluate the fourth sum in (\ref{|R_n|}) our inequality
 for $|R_n|$ becomes
\begin{equation*}
\begin{split}
|R_n|&\ll \sum_{1\leqslant j \leqslant n/2}|S(g;j)|
\frac{p(e^{-1/(n-j)})}{n-j}  \left( {\frac{j^{\theta
-1}}{(n-j)^\theta}} \right) {\frac1{p(e^{-1/j})}}
\\
&\quad+\sum_{1\leqslant j \leqslant n/2}|S(g;j)|  p_n \left( \frac{j}{n} \right)^{\theta} {\frac1{jp(e^{-1/j})}}+{\frac1{n^2}}\sum_{n/2\leqslant j \leqslant
n-1}|S(g;j)|
\\
&\quad +p_n\sum_{ j > n/2}|S(g;j)| \frac{e^{-j/n}}{jp(e^{-1/n})}
\\
&\ll p_n \left( {\frac1{n^\theta}}\sum_{j=1}^n {\frac{|S(g;j)|}{p(e^{-1/j})}}j^{\theta -1}
+ {\frac1{p(e^{-1/n})}}\sum_{j>n}{\frac{|S(g;j)|}{j}}e^{-j/n}\right).
\end{split}
\end{equation*}
The theorem is proved.
\end{proof}

\begin{proof}[ Proof of theorem \ref{voronmean}.]

{\it 1) Sufficiency.} The
 second condition (\ref{second_cond}) of the theorem means that $S(g;n)=o(np_n)$.
 Applying this asymptotic  to the right hand side of the inequality
 of  Theorem \ref{fundthm1} we get an estimate
$$
{\frac1{p_n}}\sum_{k=0}^{n}a_kp_{n-k} =g(e^{-1/n})+o(1).
$$
By the first condition (\ref{first_cond}) of the theorem $g(x)\to A$
as $x\uparrow 1$, which means that the right hand of the above
estimate converges to $A$ as $n\to \infty$. This  proves that
conditions (\ref{first_cond}) and (\ref{second_cond}) of the theorem
imply that
\begin{equation}
\label{limit_of_sum}
 \lim_{n\to \infty}{\frac1{p_n}}\sum_{k=0}^{n}a_kp_{n-k} =A.
\end{equation}
{\it 2) Necessity.} Suppose now that limit (\ref{limit_of_sum})
exists. Let us denote
$$
r_n=\sum_{k=0}^{n}a_kp_{n-k}.
$$
Our assumption (\ref{limit_of_sum}) means that
$r_n=Ap_n(1+\varepsilon_n )$, where $\varepsilon_n \to \infty$ as
$n\to \infty$.

We can express the generating function of $S(g;n)$ in terms of
generating functions of quantities $r_n$ and $d_n$ as
\begin{equation*}
\begin{split}
\sum_{j=1}^{\infty}S(g;j)z^{j}&=zg'(z)p(z)=z(p(z)g(z))'-zp'(z)g(z)
\\
&=z(p(z)g(z))'-p(z)g(z)\sum_{k=1}^{\infty}d_kz^k
\\
&=\sum_{n=1}^{\infty}nr_nz^n-\sum_{n=1}^{\infty}r_nz^n
\sum_{k=1}^{\infty}d_kz^k.
\end{split}
\end{equation*}
Equating the coefficients in the Taylor expansion of the series on
the both sides of the above equation we obtain an equation
expressing $S(g;n)$ in terms of  $r_j$ as
\[
S(g;n)=nr_n-\sum_{k=1}^nd_kr_{n-k}.
\]
Inserting here $r_n=Ap_n(1+\varepsilon_n )$ we obtain
\begin{equation*}
\begin{split}
S(g;n)&=Anp_n-A\sum_{k=1}^nd_kp_{n-k}+Anp_n\varepsilon_n-
A\sum_{k=1}^nd_kp_{n-k}\varepsilon_{n-k}
\\
&=Anp_n\varepsilon_n-
A\sum_{k=1}^nd_kp_{n-k}\varepsilon_{n-k}=o(np_n),
\end{split}
\end{equation*}
since $np_n=\sum_{k=1}^nd_kp_{n-k}$ and $np_n\gg p(e^{-1/n})\to
\infty$ as $n\to \infty$.

The necessity of condition 2) is proved.

The necessity of condition 1) is well known, see e. g.
\cite{hardy_divergent}. It is obtained by noticing that
$$
g(x)=\frac{g(x)p(x)}{p(x)}
=\frac{r_0+r_2x+\cdots+r_nx^n+\cdots}{p_0+p_2x+\cdots+p_nx^n+\cdots}.
$$
Using estimate $r_n=Ap_n(1+o(1))$ to evaluate the right hand side of
the above equation  and taking into
account the  fact that
$p(x)\to \infty$ as $x\nearrow 1$ we conclude that the left hand side
of the above equation has a limit
$$
g(x)\to A
$$
as $x\nearrow 1$.

The theorem is proved.
\end{proof}

\subsection{Random permutations}
Recall that if we denote by
 $\alpha_k(\sigma)$  the number of cycles in permutation $\sigma$ whose length is
 equal to $k$ then the value of  multiplicative function $f(\sigma)$
 can be expressed as a product
$$
f(\sigma)=\hat f(1)^{\alpha_1 (\sigma)} \hat f (2)^{\alpha_2
(\sigma)} \ldots \hat f(n)^{\alpha_n(\sigma)},
$$
where we assume that $0^0=1$ in the above relationship.

It can be proved by elementary combinatorial arguments (see e.g.
\cite{comtet_1974}, page 233, Theorem B) that the quantity of
permutations $\sigma \in S_n$ such that $\alpha_j(\sigma)=k_j$ for
$1\leqslant j \leqslant n$ is equal to
$$
n!\prod_{j=1}^{n} {\frac{1}{k_j!j^{k_j}}},
$$ when
$k_1+2k_2+\cdots +nk_n=n$.   This fact allows us to express the sum
of values $f(\sigma )$ of  multiplicative function $f$ over all
permutations $\sigma\in S_n$  as
$$
\sum_{\sigma \in S_n}f(\sigma )= n!
 \sum_{k_1+2k_2+\cdots +nk_n=n} {\prod_{j=1}^{n}}
          {     {     \left(     \frac{ \hat f(j)}{j}     \right)   }^{k_j}    }
           {\frac1{k_j!}}.
$$
The above expression leads to the identity of the corresponding
generating functions
\begin{equation}
\label{genfunc} 1+\sum_{m=1}^{\infty}\Bigl({\frac1{m!}}\sum_{\sigma
\in S_m}f(\sigma )\Bigr)z^m=\prod_{j\geqslant 1}\sum_{k=0}^\infty \biggl(\frac{\hat
f(j)}{j}\biggr)^k\frac{z^{jk}}{k!}=\exp \left\{ \sum_{j=1}^{\infty} {{\frac{\hat
f(j) }{j}}z^j}
         \right\}.
\end{equation}
This in its turn leads to the expression for the mean value of a
multiplicative function
$$
M_n^d(f)= \frac{\sum_{\sigma \in S_n}f(\sigma )d(\sigma
)}{\sum_{\sigma \in S_n}d(\sigma )}=\frac{[z^n]\exp \left\{
 \sum_{j=1}^{\infty} {{\frac{\hat{d}(j)\hat
f(j)}{j}}z^j}
         \right\}}{[z^n]\exp \left\{
 \sum_{j=1}^{\infty} {{\frac{\hat
d(j)}{j}}z^j}
         \right\}}
$$
as a ratio of coefficients of appropriate generating functions. Let
us denote the generating functions appearing in the above expression
of the mean value $M_n^d(f)$ as
\[
M(z)=\sum_{j=0}^{\infty}M_jz^j=\exp \left\{
 \sum_{j=1}^{\infty} {{\frac{\hat{d}(j)\hat
f(j)}{j}}z^j}
         \right\}
\]
and
\[
p(z) = \sum_{j=0}^{\infty}p_jz^j=\exp \left\{
 \sum_{j=1}^{\infty} \frac{\hat{d}(j)}{j}z^j \right\}.
\]
Since later we will apply Theorem \ref{fundthm1} with
$\hat{d}(j)=d_j$, so from now on we will  identify
 the quantity $\hat{d}(j)$  with $d_j$. We also will use the same
notation $p(z)$ for the function $\exp \left\{
 \sum_{j=1}^{\infty} \frac{\hat{d}(j)}{j}z^j \right\}$ as  for the
  for function $\exp \left\{
 \sum_{j=1}^{\infty} \frac{d_j}{j}z^j \right\}$.

We can represent the generating function $M(z)$ as a product of two
functions
\begin{equation}
\label{genfunc2}
\begin{split} M(z)&
=\exp \left\{\sum_{j=1}^{\infty} {{\frac{\hat{d}(j)}{j}}z^j}
\right\} \exp \left\{\sum_{j=1}^{\infty} \hat{d}(j){\frac{\hat
f(j)-1}{j}}z^j \right\}
\\
&= p(z)m(z) ,
\end{split}
\end{equation}
where
\[
m(z)=\sum_{j=0}^{\infty}m_jz^j=\exp \left\{\sum_{j=1}^{\infty}
\hat{d}(j){\frac{\hat f(j)-1}{j}}z^j \right\}.
\]
With these notations we can express the mean value of a multiplicative
function as
$$
M_n^d(f)=\frac{[z^n]p(z)m(z) }{[z^n]p(z)}=\frac{M_n}{p_n}={\frac{1}{p_{n}}}\sum_{j=0}^n m_jp_{n-j}.
$$

Before going into the details of proof of Theorem \ref{meanf} let us
 at first illustrate the ideas of application of Tauberian theory
for Voronoi summation for the analysis of mean values $M_n^d(f)$ on
a simple example.
\begin{prop}
\label{prop_simple_mult_f} Suppose we are given a fixed sequence  $\hat{f}(1),\hat{f}(2),\ldots,\hat{f}(n),\ldots$ of
complex numbers,
such that the modulus of the members of this sequence does not exceed
one $|\hat{f}(j)|\leqslant 1$ and
\[
\lim_{k\to \infty}\hat{f}(k)=1.
\]
Then the sequence of mean values of the corresponding multiplicative
function $f(\sigma)$  has the following asymptotic
\[
M_n^d(f)= \frac{\sum_{\sigma \in S_n}f(\sigma )d(\sigma
)}{\sum_{\sigma \in S_n}d(\sigma )}=\exp\left\{\sum_{j=1}^{\infty}
\hat{d}(j){\frac{\hat f(j)-1}{j}}e^{-j/n}\right\}+o(1),
\]
as $n\to \infty$.
\end{prop}
\begin{proof} We have already shown that the generating function
$M(z)$ is  a product of two functions $M(z)=p(z)m(z)$. By our
results on Voronoi means, if we can prove that Tauberian condition
\begin{equation}
\label{tauberian_cond_for_m}
[z^{n-1}]p(z)m'(z)=\sum_{k=0}^nkm_kp_{n-k}=o(np_n)
\end{equation}
holds, then
\[
[z^n]p(z)m(z)=p_n\bigl(m(e^{-1/n})+o(1)\bigr)\quad\hbox{as}\quad
n\to \infty.
\]
which will imply an estimate for the mean value
\[
M_n^d(f)=\frac{[z^n]p(z)m(z) }{[z^n]p(z)}=m(e^{-1/n})+o(1).
\]
To check the Tauberian condition (\ref{tauberian_cond_for_m}) we
note that
\[
p(z)m'(z)=p(z)m(z)\sum_{j=1}^{\infty} \hat{d}(j){\frac{\hat
f(j)-1}{j}}z^{j-1}=M(z)\sum_{j=1}^{\infty} \hat{d}(j){\frac{\hat
f(j)-1}{j}}z^{j-1},
\]
which means that
\begin{equation}
\label{expression_for_S(m,n)}
S(m;n)=[z^{n-1}]p(z)m'(z)=\sum_{j=1}^n\hat{d}(j)\bigl(\hat
f(j)-1\bigr)M_{n-j}.
\end{equation}
Since $\bigl|\hat{f}(j)\bigr|\leqslant 1$ we have
\[
\begin{split}
|M_k|&=\left|[z^k]\exp\left\{\sum_{j=1}^\infty\frac{d_j\hat{f}(j)}{j}z^j\right\}\right|=
\left|\sum_{\ell_1+2\ell_2+\cdots +k\ell_k=k} {\prod_{j=1}^{k}}
          {     {     \left(     \frac{ d_j\hat  f(j)}{j}     \right)   }^{\ell_j}    }
           {\frac1{\ell_j!}}\right|
\\
&\leqslant\sum_{\ell_1+2\ell_2+\cdots +k\ell_k=k} {\prod_{j=1}^{k}}
          {     {     \left(     \frac{  d_j }{j}     \right)   }^{\ell_j}    }
           {\frac1{\ell_j!}}
=
[z^k]\exp\left\{\sum_{j=1}^\infty\frac{d_j}{j}z^j\right\}=p_k.
\end{split}
\]
Applying inequalities $|M_k|\leqslant p_k$ and $d_j=\hat{d}(j)\leqslant d^+$ to estimate
 the right hand side of the identity (\ref{expression_for_S(m,n)}) we conclude that in order to show that
$S(m;n)=o(np_n)$ it is enough to check that
\begin{equation}
\label{sum_t}
\sum_{j=1}^n\bigl|\hat f(j)-1\bigr|p_{n-j}=o(np_n).
\end{equation}
From the recurrence relations (\ref{p_n}) for $p_n$ we obtain an
inequality
\[
p_j=\frac{1}{j}\sum_{k=1}^jd_kp_{j-k}\leqslant \frac{d^+}{j}\sum_{k=0}^{j-1}p_{k}\leqslant
\frac{d^+}{d^-j}\sum_{k=1}^nd_kp_{n-k}=\frac{d^+n}{d^-j}p_n
\]
for any $1\leqslant j\leqslant n$. This inequality enables us to
show that a finite number of  the first summands of the  sum of
(\ref{sum_t}) is negligible. That is, for any fixed $T<n$ the
inequality
\[
\begin{split}
\sum_{j=1}^n\bigl|\hat f(j)-1\bigr|p_{n-j}&\leqslant
\frac{2d^+np_n}{d^-}\sum_{j\leqslant T} \frac{1}{n-j}
+\sup_{j\geqslant T}\bigl|\hat f(j)-1\bigr|\sum_{j=1}^np_{n-j}
\\
&\leqslant  \frac{2np_nd^+T}{d^-(n-T)}+\sup_{j\geqslant T}\bigl|\hat
f(j)-1\bigr|\frac{np_n}{d^-}
\end{split}
\]
holds. By the condition of our proposition $\sup_{j\geqslant
T}\bigl|\hat f(j)-1\bigr|\to 0$ as $T\to \infty$. Therefore  if we
chose, for example $T:=\sqrt{n}$, then  the second term of the sum
in the last inequality will be $o(np_n)$ while the first term  will
be of order $O(\sqrt{n}p_n)=o(np_n)$.  This proves that sum of both
terms is $o(np_n)$.
\end{proof}
A number of conclusions can be drawn from the just proven
proposition. First of all it is clear that if multiplicative
function $f(\sigma)$ satisfies the conditions of the proposition,
then the mean value has a zero limit
\[
\lim_{n\to\infty}M_n^d(f)=0
\]
if and only if diverges the series
\[
\sum_{j=1}^\infty \frac{1-\Re \hat{f}(j)}{j}=\infty.
\]
The case of convergence of the above series can  be split into
two cases. The existence of the non-zero limit
$\lim_{n\to\infty}M_n^d(f)\not=0$ is equivalent to convergence of
the series
\[
\sum_{j=1}^\infty d_j\frac{\hat{f}(j)-1}{j}.
\]
If the imaginary part of the above series diverges then
$M_n^d(f)=c(f) e^{iL(n)}+o(1)$, where $L(n)$ is a real and slowly
varying function and $c(f)=e^{\sum_{j=1}^\infty d_j\frac{\Re
\hat{f}(j)-1}{j}}$ is a positive constant.

The result of Proposition \ref{prop_simple_mult_f} and its
consequences are not new. They follow from more general results of
Manstavi\v{c}ius \cite{manst_decomposable_2002} that were obtained
by using a different approach.
\begin{proof}[Proof of Theorem \ref{meanf}] Since $\hat f(k)$ for $k>n$ do not influence the $n$-th Taylor coefficient $M_n$ of the generating
function $M(z)$,
 we assume that $\hat f(k)=1$ for $k>n$.
In the proof of the previous Proposition we have already obtained
inequality
$$
|S(m;j)|=\bigl|[z^{j-1}]p(z)m'(z)\bigr|\leqslant d^+ \sum_{k=1}^{j}|\hat f(k) -1|p_{j-k}.
$$
 Using this estimate to bound the right hand side of the inequality
 of Theorem \ref{fundthm1} with $g(z)=m(z)$
we obtain inequality
\begin{multline}
\label{inequality_mean} \left| {\frac{M_n}{p_n}}- m(e^{-1/n})
\right|=\left| {\frac{M_n}{p_n}}- \exp \left\{
\sum_{k=1}^{n}{\frac{d_k(\hat f(k)-1)}{k}}e^{-k/n} \right\} \right|
\\
\leqslant c \left( {\frac1{np_n}}\sum_{k=1}^{n}|\hat f(k) -1|p_{n-k}+
{\frac1{n^\theta }}\sum_{j=1}^n{\frac{j^{\theta -1}}{p(e^{-1/j})}}
\sum_{k=1}^{j}|\hat f(k) -1|p_{j-k} \right.
\\
\left. \quad\mbox{}+ {\frac1{p(e^{-1/n})}}
\sum_{j>n}{\frac{e^{-j/n}}j}\sum_{k=1}^{j}|\hat f(k) -1|p_{j-k} \right) ,
\end{multline}
where  $\theta =\min \{ 1, d^- \}$ and $c=c( d^+, d^- )$. The
inequality in the formulation of the theorem will follow after we
 find simpler estimates for the second and the third term in
the sum on the right hand side of the above inequality. Let us start
with the second term. Changing the order of summation we get
\begin{equation*}
\begin{split}
&{\frac1{n^\theta }}\sum_{j=1}^n{\frac{j^{\theta -1}}{p(e^{-1/j})}}
\sum_{k=1}^{j}|\hat f(k) -1|p_{j-k}= {\frac1{n^\theta
}}\sum_{k=1}^{n}|\hat f(k)-1| \sum_{n\geqslant j\geqslant
k}{\frac{j^{\theta -1}}{p(e^{-1/j})}}p_{j-k}
\\
&\ll {\frac1{n^\theta }}\sum_{k=1}^{n}|\hat f(k)-1| \sum_{n\geqslant
j\geqslant 2k}{\frac{j^{\theta
-1}}{p(e^{-1/j})}}{\frac{p(e^{-1/(j-k)})}{j-k}}
\\
&\quad+
 {\frac{1}{n^\theta }}\sum_{k=1}^{n}|\hat f(k)-1| k^{\theta -1}\sum_{2k\geqslant
j\geqslant k}{\frac{p_{j-k}}{p(e^{-1/j})}}.
\end{split}
\end{equation*}
As a Taylor series with positive coefficients function $p(x)$
is increasing for increasing values of $x$.
Therefore for $j>k$ we have
${p(e^{-1/(j-k)})}\leqslant{p(e^{-1/j})}$ and $p(e^{-1/j})\geqslant p(e^{-1/k})$.
Using these inequalities
to evaluate the  last estimate we get
\[
\begin{split}
&{\frac1{n^\theta }}\sum_{j=1}^n{\frac{j^{\theta -1}}{p(e^{-1/j})}}
\sum_{k=1}^{j}|\hat f(k) -1|p_{j-k}
\\
&\ll {\frac{1}{n^\theta}}\sum_{k=1}^{n}|\hat f(k)-1|\sum_{n\geqslant
j \geqslant 2k} j^{\theta -2}+{\frac1{n^\theta}}\sum_{k=1}^{n}|\hat
f(k)-1|k^{\theta -1} {\frac1{p(e^{-1/k})}}\sum_{s=0}^k p_s
\\
&\ll {\frac1{n^\theta}}\sum_{k=1}^{n}|\hat f(k)-1|\left( k^{\theta
-1} +\int_k^nx^{\theta-2}\,dx \right)
\end{split}
\]
The third term can be handled in a similar way
\begin{multline*}
{\frac1{p(e^{-1/n})}}\sum_{j>n}{\frac{e^{-j/n}}j}\sum_{k=1}^n|\hat f(k)-1|p_{j-k}={\frac1{p(e^{-1/n})}}\sum_{k=1}^n|\hat f(k)-1|\sum_{j>n}{\frac{e^{-j/n}}j}p_{j-k}
\\
\leqslant \frac1n\sum_{k=1}^n|\hat f(k)-1|{\frac{e^{-k/n}}{p(e^{-1/n})}}
\sum_{j>n}{{e^{-(j-k)/n}}}p_{j-k}\leqslant {\frac1n}\sum_{k=1}^n|\hat f(k)-1|.
\end{multline*}
The statement of the theorem will follow if we use inequality
\[
\left| \sum_{k=1}^{n}{\frac{d_k(\hat f(k)-1)}{k}}e^{-k/n} -
\sum_{k=1}^{n}\frac{d_k(\hat f(k)-1)}{k}\right| \leqslant
\frac{1}{n}\sum_{k=1}^{n}d_k|\hat f(k)-1|
\]
to estimate the quantity under exponent in our inequality
(\ref{inequality_mean}).

 The theorem is proved.
 \end{proof}

Let us define
$$
L_n(z)=\sum_{j=1}^nd_j{\frac{\hat f(j)-1}{j}}z^j \quad \hbox{and}
\quad  \rho_n (p)=\left( \sum_{j=1}^n{\frac{|\hat f(j)-1|^p}{j}}
\right)^{1/p},
$$
moreover we assume that
$$
\rho_n(\infty)=\lim_{p\to \infty}\rho_n (p)=\max_{1\leqslant k
\leqslant n}|\hat f(k) -1|\quad\hbox{and}\quad \rho(p)=\left(
\sum_{j=1}^\infty{\frac{|\hat f(j)-1|^p}{j}} \right)^{1/p},
$$
likewise we will write $L(z)=\sum_{j=1}^\infty d_j{\frac{\hat
f(j)-1}{j}}z^j$. For brevity, we will often write simply $\rho$
instead of $\rho(p)$ if the value of $p$ is known from the context.
\begin{lem}
\label{diffLmn} For $n,m\geqslant 1$ and $1/p+1/q=1$ with $\infty
\geqslant p>1$, we have
\[
 |L(e^{-1/n})-L(e^{-1/m})|\leqslant d^+ \rho (p) \left(1+\left|\log\frac{n}{m}\right|\right)
\]
and
\[
|L(e^{-1/n})|\leqslant d^+ \rho (p)(1+\log n).
\]
\end{lem}
\begin{proof} Without loss of generality we can assume that $n\geqslant m\geqslant 1$.
Suppose first that $p<\infty$. Applying
Cauchy's  inequality with parameters $p,q$ we have
\[
\begin{split}
 |L(e^{-\frac1n})-L(e^{-\frac1m})|
 &\leqslant d^+\sum_{j=1}^\infty\frac{|\hat f (j)-1|}{j}|e^{-j/n}-e^{-j/m}|
 \\
 &\leqslant d^+\rho(p)\left(\sum_{j=1}^\infty
 \frac{|e^{-j/n}-e^{-j/m}|^q}{j}\right)^{1/q}
 \\
 &\leqslant d^+\rho(p)\left(\sum_{j=1}^\infty
 \frac{e^{-j/n}-e^{-j/m}}{j}\right)^{1/q}
 \\
 & =d^+\rho(p)\left(\log\frac{1-e^{-1/m}}{1-e^{-1/n}}\right)^{1/q}
 \\
 & \leqslant d^+\rho(p)\left(\frac{1}{n}+\log\frac{n}{m}\right)^{1/q}.
\end{split}
\]
in the last step we used inequality $e^{-x}x<1-e^{-x}<x$, which is
true for all $x>0$, to estimate the ratio of quantities
$\frac{1-e^{-1/m}}{1-e^{-1/n}}$ under the sign of logarithm. Finally
estimating the fraction $1/m$ in the last expression by means of a
crude upper bound $\frac1m\leqslant 1$ and taking
 into account that $1/q\leqslant 1$ we obtain the first inequality in the statement of the lemma for
finite values of $p>1$.
Allowing $p\to \infty$ we see that this inequality is true for
$p=\infty$
 also. Similar considerations lead to the second inequality.
\end{proof}

 \begin{lem}
 \label{sumpn}
 For all $n\geqslant 1$ and $\varepsilon\geqslant 0$, $q\geqslant 1$ fixed  such that $q(d^-
-1)-\varepsilon>-1$, then
$$
\sum_{j=1}^nj^{-\varepsilon} p_j^q\ll n^{1-\varepsilon}p_n^q    \quad \hbox{and}\quad \sum_{j=1}^{n}
                {\frac1j}
       {\left| { \frac{p_{n-j}}{p_{n}} }
       -1 \right|}^q\ll 1.
$$
\end{lem}
\begin{proof} Applying the upper  (\ref{p_n2}) and lower (\ref{plwbound}) bounds for $p_j$ and using the estimate
for ratio $\frac{p(e^{-1/j})}{p(e^{-1/n})}$ provided by Lemma \ref{pvbound} we obtain
\[
\begin{split}
\sum_{j=1}^nj^{-\varepsilon}p_j^q&\leqslant \sum_{j=1}^n\frac{1}{j^{\varepsilon}}\left(ed^+\frac{p(e^{-1/j})}{j}\right)^q
= (ed^+)^q\left(\frac{p(e^{-1/n})}{n}\right)^q
\sum_{j=1}^n\frac{1}{j^{\varepsilon}}\left(\frac{np(e^{-1/j})}{jp(e^{-1/n})}\right)^q
\\
&\ll  \left(\frac{p(e^{-1/n})}{n}\right)^q\sum_{j\leqslant n}
\frac{1}{j^{\varepsilon}}\left(  \frac{j}{n} \right)^{q (d^--1)}
\ll n^{1-\varepsilon} p_n^q.
\end{split}
\]
We will estimate the second sum by splitting it into two parts
$$
\sum_{j=1}^{n} {\frac1j}
       {\left|  \frac{p_{n-j}}{p_{n}} -1 \right|}^q\ll
\sum_{1\leqslant j\leqslant n/4} \frac1j
       {\left|  \frac{p_{n-j}-p_n}{p_{n}}  \right|}^q+
{\frac1n}\sum_{n/4<j\leqslant n }\left(
       {\left| { \frac{p_{n-j}}{p_{n}} } \right|}^q+1\right).
$$
The second sum in the last inequality is $O(1)$ by the just proven
estimate for the sum of  $p_j^q$. While the first sum can be
estimated applying the upper bound  for difference $p_{n-j}-p_n$
provided by Lemma \ref{diffpnm}, which yields
\begin{equation*}
\begin{split}
\sum_{j\leqslant n/4} {\frac1j}
       {\left| { \frac{p_{n-j}-p_{n}}{p_{n}} } \right|}^q&\ll
\sum_{1\leqslant j\leqslant n/4} {\frac1j}
       {\left|
       {\frac jn} \right|}^{q\theta}\ll 1.
\end{split}
\end{equation*}
The lemma is proved.
\end{proof}

\begin{prop}
\label{prop_O(rho)} Consider a sequence of complex numbers
$\hat{f}(1),\hat{f}(2),\ldots,\hat{f}(n),\ldots$, such that
$|\hat{f}(j)|\leqslant 1$ then for any fixed $p>1$ there is a
positive constant $c_3=c_3(p,d^-,d^+)$ such that
\[
\left|\frac{M_n}{p_n}-m(e^{-1/n})\right|\leqslant
c_3\left(\sum_{j=1}^\infty\frac{|\hat{f}(j)-1|^p}{j}\right)^{1/p}\quad\bigl(=c_3\rho(p)\bigr)
\]
for all $n\geqslant 1$.
\end{prop}
\begin{proof} If the series on the right hand side of the inequality of
our proposition diverges, then the proposition becomes trivial.
Therefore let us assume that this series is convergent. Applying the
estimate of sums of $p_j^q$ of Lemma \ref{sumpn} we obtain
inequality
\[
\begin{split}
|S(m;n)|&\leqslant\sum_{j=1}^n\hat{d}(j)\bigl|\hat
f(j)-1\bigr|p_{n-j}\leqslant  d^+
\left(\sum_{j=1}^np_{n-j}^q\right)^{1/q}\left(\sum_{j=1}^n|\hat{f}(j)-1|^p\right)^{1/p}
\\
&\ll np_n\left(\sum_{j=1}^n\frac{|\hat{f}(j)-1|^p}{n}\right)^{1/p}
\leqslant np_n\left(\sum_{j=1}^\infty\frac{|\hat{f}(j)-1|^p}{j}\right)^{1/p}
\end{split}
\]
Applying this inequality to estimate the right hand side of inequality Theorem \ref{fundthm1} for difference $\frac{M_n}{p_n}-m(e^{-1/n})$, we complete the proof of the Proposition.
\end{proof}
Unfortunately the estimate of the just proven proposition is not
strong enough for our purpose as we will need an estimate like
$O\bigl(\rho(p)m(e^{-1/n})\bigr)$ in order to analyze characteristic
functions of additive functions.
\begin{prop}
\label{prop_O(rho_exp)} For any fixed $\infty \geqslant p>\max
\left\{ 1, 1/d^- \right\}$, there exists such a positive
$\delta=\delta (d^-,d^+,p)$ that if $\rho=\rho(p) \leqslant \delta$,
then
\[
\frac{M_n}{p_n}=m(e^{-1/n})\bigl(1+O(\rho)\bigr).
\]
\end{prop}
\begin{proof}
At first let us prove that $M_n=O(p_n|m(e^{-1/n})|)$ for all
$n\geqslant 0$ (here and in what follows we assume that
$e^{-1/0}=0$). Let us assume that only finite number of $\hat{f}(j)$
are not equal to $1$. Then the supremum $D$ of ratios
\[
D:=\sup_{n\geqslant 0}\left|\frac{M_n}{p_nm(e^{-1/n})}\right|
\]
will be finite. We will prove that $D\leqslant 2$ if
$\rho(p)\leqslant \delta$, with some absolute, sufficiently small
constant $\delta>0$. Let us use inequality $|M_k|\leqslant D
p_k|m(e^{-1/k})|$ to estimate the right hand side of the identity
(\ref{expression_for_S(m,n)}) for $S(m;n)$ as
\[
\begin{split}
|S(m;n)|&\leqslant Dd^+\sum_{j=1}^n\bigl|\hat
f(j)-1\bigr|p_{n-j}\bigl|m(e^{-1/{(n-j)}})\bigr|
\\
&\leqslant Dd^+n\left(\frac{1}{n}\sum_{j=1}^n\bigl|\hat
f(j)-1\bigr|^p\right)^{1/p}\left(\frac{1}{n}\sum_{j=0}^{n-1}p_{j}^q\bigl|m(e^{-1/j})\bigr|^q\right)^{1/q}
\\
&\leqslant Dd^+n\rho\bigl|m(e^{-1/n})\bigr|
\left(\frac{1}{n}\sum_{j=0}^{n-1}p_{j}^q\left|\frac{m(e^{-1/j})}{m(e^{-1/n})}\right|^q\right)^{1/q},
\end{split}
\]
for $n\geqslant 1$. Since $m(e^{-1/k})=e^{L(e^{-1/k})}$ we can apply
Lemma \ref{diffLmn} to estimate the ratio
\begin{equation}
\label{upper_bound_ratio}
\begin{split}
\left|\frac{m(e^{-1/j})}{m(e^{-1/n})}\right|&\leqslant
\exp\bigl\{|L(e^{-1/n})-L(e^{-1/j})|\bigr\}
\\
&\leqslant \exp\left\{ d^+ \rho(p)
\left(1+\log\frac{n}{j}\right)\right\}=e^{d^+\rho(p)
}\Bigl(\frac{n}{j\vee 1}\Bigr)^{d^+ \rho(p)},
\end{split}
\end{equation}
for $n\geqslant j\geqslant 0$, where we use the notation $a\vee
b=\max\{a,b\}$. From now on let us assume that $\delta$ is small
enough  that $q(d^--1)-qd^+\delta>-1$, which is necessary to ensure
the validity of the upper bound of Lemma \ref{sumpn} for the partial
sum of $p^q_jj^{-qd^+\rho(p)}$. This allows us to further evaluate
$S(m;n)$ as
\[
\begin{split}
|S(m;n)|&\leqslant Dd^+e^{\rho d^+}n\rho\bigl|m(e^{-1/n})\bigr|
\left(\frac{1}{n}\sum_{j=0}^{n-1}p_{j}^q\Bigl(\frac{n}{j\vee
1}\Bigr)^{qd^+ \rho(p)}\right)^{1/q}
\\
&\leqslant C_1D\rho n p_n\bigl|m(e^{-1/n})\bigr|,
\end{split}
\]
 where
$C_1=C_1(d^+,d^-,p)$ is a positive constant. Plugging this estimate into
the right hand side of the inequality of Theorem \ref{fundthm1}
 we obtain
\begin{equation*}
\begin{split} \left|\frac{M_n}{p_n}-m(e^{-1/n})\right|&
\ll \frac{D\rho}{n^{\theta}}\sum_{j=1}^n
\bigl|m(e^{-1/j})\bigr|j^{\theta-1}
+\frac{D\rho}{p(e^{-1/n})}\sum_{j>n}\bigl|m(e^{-1/j})\bigr|p_je^{-j/n}
\\
&\quad+\rho D\bigl|m(e^{-1/n})\bigr|,
\end{split}
\end{equation*}
and once again utilizing the upper bound (\ref{upper_bound_ratio})
for ratio $\left|\frac{m(e^{-1/j})}{m(e^{-1/n})}\right|$ and
noticing that $\bigl|m(e^{-1/j})\bigr|$ is monotonously decreasing
as $j$ increases  we finally get
\begin{equation}
\label{eq_proof} \left|\frac{M_n}{p_n}-m(e^{-1/n})\right|\leqslant
C_2D\rho(p) \bigl|m(e^{-1/n})\bigr|,
\end{equation}
if $\delta d^+<\theta$, where again $C_2$ is a positive constant
that depends on $d^+$, $d^-$ and $p$ only. Dividing both sides of
this inequality by $\bigl|m(e^{-1/n})\bigr|$, taking maximum for all
$n\geqslant 0$ and recalling the definition of $D$ we conclude that
this quantity satisfies inequality
\[
D\leqslant 1+C_2D\rho(p).
\]
Thus if we require $\delta$ to be fixed and small enough to ensure
that $\rho(p) \leqslant \delta \leqslant 1/(2C_2)$, then $D$ would
be bounded $D\leqslant 2$. Thus if $\delta$ is fixed such that
$\delta
<\min\left\{\frac{q(d^--1)+1}{qd^+},\frac{\theta}{d^+},\frac{1}{2C_2}\right\}$
then the estimate  of  the proposition will follow from the inequality
(\ref{eq_proof}) and the fact that $D$ is bounded for such $\delta$.
Note that at the beginning of the proof we assumed that only a finite
number of $\hat{f}(j)$  are not equal to $1$. However this condition
was only needed to ensure that quantity $D$ is finite, all the
constants in symbols $O(\ldots)$ and $\ll$ do not depend on the
number of $\hat{f}(j)$ that are not equal to $1$. Thus if we have an
infinite sequence $\hat{f}(j)$ such that $\rho(p)\leqslant \delta$
then we can consider a modified sequence that is obtained by putting
$\hat{f}(j)=1$ for $j\geqslant N$  and allow $N\to \infty$.

\end{proof}

The following theorem has been proved for $d_j\equiv 1$ by
Manstavi\v cius \cite{manst_additive_1996}, later generalized for
$d_j\equiv \theta >0$ in \cite{zakh_palanga_2001}.
\begin{thm}
\label{meanM1} For any fixed $\infty \geqslant p>\max \left\{ 1,
1/d^- \right\}$, there exists such a positive $\delta=\delta
(d^-,d^+,p)$ that if $\rho \leqslant \delta$, then

\begin{equation}
\label{manst_general} \frac{M_N}{p_N}=\exp \{  L_N(1) \} \left(
1+ \sum_{j=1}^Nd_j{\frac{\hat f(j)-1}j} \left( {\frac{p_{N-j}}{p_{N}}}-1 \right) +O(\rho^2)\right).
\end{equation}

\end{thm}
\begin{proof} The values of $\hat{f}(j)$ with $j>N$ do not influence
the value of $M_n$ therefore  we will   assume that $\hat{f}(j)=1$
for all $j>N$. Let us consider $S(g;n)$ with
$g(z)=u_N(z)=m(z)-m(e^{-1/N}) \sum_{j=1}^\infty d_j{\frac{\hat
f(j)-1}j}z^j$ instead of $g(z)=m(z)$ then
\[
\begin{split}
S(u_N;n)&=[z^{n-1}]p(z)u_N'(z)
\\
&=S(m;n)-[z^{n-1}]m(e^{-1/N})p(z) \sum_{j=1}^\infty d_j{\bigl(\hat
f(j)-1\bigr)}z^{j-1}
\\
&=\sum_{j=1}^nd_j\bigl(\hat
f(j)-1\bigr)\bigl(M_{n-j}-m(e^{-1/N})p_{n-j}\bigr)
\end{split}
\]
Applying here the estimate
$M_{n-j}=p_{n-j}m(e^{1/(n-j)})\bigl(1+O(\rho)\bigr)$ of Proposition
\ref{prop_O(rho_exp)} after some evaluations we get
\[
\begin{split}
S(u_N;n)&\ll\sum_{j=1}^n\bigl|\hat
f(j)-1\bigr|\bigl|m(e^{-1/{(n-j)}})-m(e^{-1/N})\bigr|p_{n-j}
\\
&\quad+ \rho\sum_{j=1}^n\bigl|\hat
f(j)-1\bigr|p_{n-j}\bigl|m(e^{-1/{(n-j)}})\bigr|
\end{split}
\]
The second sum of the above estimate has already been already shown
to be $O\left(\rho^2np_n\bigl|m(e^{-1/n})\bigr|\right)$ in the proof
of Proposition \ref{prop_O(rho_exp)}. Therefore
\[
\begin{split}
S(u_N;n)&\ll
n\rho\left(\frac{1}{n}\sum_{j=0}^{n-1}\bigl|m(e^{-1/{j}})-m(e^{-1/N})\bigr|^qp_{j}^q\right)^{1/q}
+ \rho^2np_n\bigl|m(e^{-1/n})\bigr|
\\
&\ll n\rho^2\bigl|m(e^{-1/N})\bigr|\left(\frac{1}{n}\sum_{j=0}^{n-1}
\exp\left\{ qd^+\rho\left|\log\frac{N}{j\vee
1}\right|\right\}p_{j}^q\right)^{1/q}
+\rho^2np_n\bigl|m(e^{-1/n})\bigr|
\end{split}
\]
Since $e^{\beta|\log x|}\leqslant x^\beta+x^{-\beta}$ if $\beta>0$,
we can further estimate
\[
\begin{split}
S(u_N;n)&\ll n\rho^2\left(\frac{1}{n}\sum_{j=0}^{n-1}
p_{j}^q\left(\Bigl(\frac{N}{j\vee
1}\Bigr)^{qd^+\rho}+\Bigl(\frac{j}{N}\Bigr)^{qd^+\rho}\right)\right)^{1/q}
+\rho^2np_n\bigl|m(e^{-1/n})\bigr|
\\
&\ll n\rho^2p_n\bigl|m(e^{-1/N})\bigr|
\left(\Bigl(\frac{N}{n}\Bigr)^{d^+\rho}+\Bigl(\frac{n}{N}\Bigr)^{d^+\rho}\right)+
\rho^2np_n\bigl|m(e^{-1/n})\bigr|.
\end{split}
\]
After plugging this estimate into the inequality (\ref{fund_ineq})
for Voronoi mean
 we end up with an estimate
\[
\frac{1}{p_N}[z^N]p(z)u_N(z)-u_N(e^{-1/N})\ll\rho^2\bigl|m(e^{-1/N})\bigr|,
\]
which after recalling the definition of $u_N$ becomes
\[
\frac{M_N}{p_N}=m(e^{-1/N})\left(1+\sum_{j=1}^Nd_j\frac{\bigl(\hat
f(j)-1\bigr)}{j}\left(\frac{p_{N-j}}{p_N}-e^{-j/N}\right)+O(\rho^2)\right).
\]
Applying here estimate
\[
m(e^{-1/N})=\exp\{{L_N(e^{-1/N})}\}=\exp\{L_N(1)\}\bigl(1+(L_N(e^{-1/N})-L_N(1))+O(\rho^2)\bigr)
\]
we complete the proof of the theorem.
\end{proof}

For $u>0$ we define
$$
E(u):=\exp \left\{ 2\sum_{{\scriptstyle k=1}\atop {\scriptstyle
|\hat f(k)-1|>u}}^n {\frac{|\hat f(k)-1|}k} \right\}.
$$
\begin{thm}
\label{thmeanM1}
There exists such a constant $\eta =\eta
(d^-,d^+)$ that for any $u\leqslant \eta$ we have
$$
\left| {{M_n}{p_n}^{-1}} \right|\ll \left| \exp \left\{  L_n(1)
\right\} \right| {\bigl( E(u) \bigr)}^{d^+ } ,
$$
\end{thm}
\begin{proof} Given a sequence of complex numbers $\hat{f}(j)$, $j\geqslant
1$ and a positive number  $u>0$ we can construct a new sequence
$\hat f_u(j)$ defined as
$$
\hat f_u(j)=\begin{cases}
1,& \text{if   $|\hat f(j)-1|>u$;} \\
\hat f(j), & \text{if  $|\hat f(j)-1|\leqslant u$.}
\end{cases}
$$
We will denote the corresponding generating function of quantities
$M_k^{(u)}$ corresponding to sequence of $\hat f_u(j)$ as
$$
M_u(z):=\exp \left\{  \sum_{j=1}^{\infty}d_j{\frac{\hat f_u(j)}{j}}z^j \right\} ={{\exp \{ L_n^{(u)}(z)\} }p(z)  }
=\sum_{k=0}^{\infty} M_k^{(u)}z^k
$$
where
$$
L_n^{(u)}(z)=\sum_{j=1}^{n}d_j{\frac{\hat f_u(j)-1}{j}}z^j=
\sum_{\substack{j\leqslant n \\|\hat f(j)-1|\leqslant u}}d_j{\frac{\hat f(j)-1}{j}}z^j.
$$
It is clear that the generating functions $M(z)$ and $M_u(z)$ are
related by identity
$$
M(z)=M_u(z)\exp \left\{  \sum_{|\hat f(j)-1|>u}d_j{\frac{\hat
f(j)-1}{j}}z^j \right\}.
$$
The above identity leads to the relation between the
coefficients in the Taylor expansion of the corresponding functions
\begin{equation}
\label{M_n_in_terms_of_m_n_u}
M_n=\sum_{k=0}^n
M_k^{(u)}m_{n-k}^{(u)},
\end{equation}
where $m_k^{(u)}$ are the coefficients in the Taylor expansion of
the generating function
$$
m^{(u)}(z)=\sum_{j=0}^{\infty} m_j^{(u)}z^j=\exp\left\{ \sum_{|\hat
f(j)-1|>u} d_j{\frac{\hat f(j)-1}{j}}z^j \right\}.
$$
Differentiating $m^{(u)}(z)$ one can easily see that $m_k^{(u)}$
satisfy the recurrent relationship
$$
m_j^{(u)}= {\frac1{j}}\sum_{{\scriptstyle 1\leqslant k \leqslant
j} \atop {\scriptstyle |\hat f(k)-1| >u}}d_k(\hat
f(k)-1)m_{j-k}^{(u)},
$$
which is true for all $j\geqslant 1$. Hence
\begin{equation*}
\begin{split}
|m_j^{(u)}|&=\left| {\frac1{j}}\sum_{{\scriptstyle 1\leqslant k
\leqslant j} \atop {\scriptstyle |\hat f(k)-1| >u}}d_k(\hat
f(k)-1)m_{j-k}^{(u)}\right| \leqslant {\frac{2d^+}j}\sum_{k=0}^{\infty}|m_k^{(u)}|\\
&\leqslant  {\frac{2d^+}j}\exp\left\{ d^+\sum_{|\hat f(\ell)-1|>u}
{\frac{|\hat f(\ell)-1|}{\ell}} \right\}={\frac{2d^+}j}E^{d^+ /2}(u).
\end{split}
\end{equation*}
Clearly, the newly formed sequence $\hat f_u(j)$ has a property that
$\bigl|\hat f_u(j)-1\bigr|\leqslant u$ for all $j\geqslant 1$.
Therefore if we assume that $u\leqslant\eta\leqslant \delta(d^-,d^+,
\infty )$ then since $\rho(\infty)\leqslant u$ the conditions of the
Proposition \ref{prop_O(rho_exp)} will be satisfied, which gives us the
estimate $M_k^{(u)}=p_km^{(u)}(e^{-1/k})(1+O(u))$. We can now use
this asymptotic of $M_k^{(u)}$ to estimate the right hand side of
the identity (\ref{M_n_in_terms_of_m_n_u}) expressing $M_n$ in terms
of $M_j^{(u)}$ and $m_{j}^{(u)}$  and obtain
\begin{equation*}
\begin{split}
|M_n|&=\sum_{k=0}^n |M_k^{(u)}m_{n-k}^{(u)}|
\\
&\ll E^{d^+/2}(u)\sum_{k\leqslant
n/2}{\frac{p_k|m^{(u)}(e^{-1/k})|}{n-k}}
+\bigl|m^{(u)}(e^{-1/n})\bigr|p_n\left(
\sum_{k=0}^{\infty}|m_k^{(u)}| \right)
\\
 &\ll E^{d^+/2}(u)\bigl|m^{(u)}(e^{-1/n})\bigr|\left( {\frac1n}\sum_{k\leqslant n/2}
p_k\frac{\bigl|m^{(u)}(e^{-1/k})\bigr|}{\bigl|m^{(u)}(e^{-1/n})\bigr|}
 +p_n \right)
\\
&\ll |\exp \{ L_n^{(u)}(1) \} | E^{d^+ /2}(u)\left( \sum_{k\leqslant
n/2} p_k{\left( \frac{n}{k} \right) }^{ud^+} +p_n \right),
\end{split}
\end{equation*}
here we have used the estimate
$\frac{\bigl|m^{(u)}(e^{-1/k})\bigr|}{\bigl|m^{(u)}(e^{-1/n})\bigr|}=|\exp
\{ (L_k^{(u)}(e^{-1/k})-L_n^{(u)}(e^{-1/n})) \} | \leqslant
{\left( \frac{n}{k}\right)}^{ud^+ }$ for $k\leqslant n$ and
estimated the sum of $p_kk^{-ud^+}$ by means of Lemma \ref{sumpn},
assuming that $\eta$ fixed such that $u\leqslant\eta <
{\frac{d^-}{d^+}}$. We finally obtain
$$
|M_n|\ll   p_n|\exp \{ L_n^{(u)}(1) \} | E^{d^+ /2}(u) \leqslant
p_n|\exp \{  L_n(1) \} | E^{d^+ }(u).
$$
Thus we have proven that the theorem holds with $\eta =\min
\{\delta(d^-,d^+,\infty ),d^+/(2d^-)\}$.
\end{proof}

 Let us now find the generating function of the characteristic
function $g(t)$ of the distribution of $h_n(\sigma)$. Notice that if
$h_n(\sigma)$ is an additive function, then $\exp\{it
h_{n}(\sigma)\}$ is a multiplicative function of $\sigma$. Therefore
$$
g_n(t)=\mathbb{E}\exp\{ith_{n}(\sigma)\}=\sum_{\sigma \in S_n}
           { \exp \left\{ ith_n(\sigma) \right\} }
           \nu_{n,d}(\sigma)
           =\frac{\sum_{\sigma \in S_n}d(\sigma )\exp \left\{ ith_n(\sigma) \right\} }{\sum_{\sigma \in S_n}
d(\sigma )}.
       $$
Thus we can apply our theorems for mean values of multiplicative
functions for multiplicative function $f$ defined by $\hat
f(k)=\exp\{it\hat h_{n}(k)\}$ to analyze the asymptotic behavior of
the characteristic function of $h_{n}(\sigma)$.

\begin{proof}[Proof of Theorem \ref{clt1}]
 Putting
 $\hat f(k)=e^{it\hat{h}_{n}(k)}$ in Theorem \ref{meanM1}, for
$|t|\leqslant \delta L_{n,p}^{-1/p}$ we have
\begin{equation}
\label{phi}
\begin{split}
\phi_n (t)&:=\int_{-\infty}^{\infty}e^{itx}\,dF_n(x)
\\
&=e^{-itA(n)}\exp \{  L_n(1) \} \left( 1+
\sum_{j=1}^nd_j{\frac{\hat{f}(j)-1}j}\left( {\frac{p_{n-j}}{p_{n}}} -1 \right) +O(\rho^2)\right)
\\
&=e^{-t^2/2+O(|t|^3L_{3,n})} \left(
1+C_nit+O\biggl(|t|^2\sum_{j=1}^nd_j{\frac{a_{nk}^2}j}\left|
{\frac{p_{n-j}}{p_{n}}} -1 \right| \biggr)+O\bigl(
|t|^2L_{n,p}^{2/p}\bigr) \right)
\\
&=e^{-t^2/2}\left( 1+C_nit+ O\bigl(
 |t|^2(1+|t|^3)(L_{n,p}^{2/p}+L_{n,3} +L'_{n,2})\bigr)
\right).
\end{split}
\end{equation}

As in \cite{manst_berry_1998}, from Theorem \ref{thmeanM1} we deduce
the existence of some sufficiently small  $c=c(d^-,d^+)$, that if
 $|t|\leqslant cL_{n,3}^{-1}=:T$ then

\begin{equation}
\label{phibound} |\phi_n (t)|\ll e^{-c_1t^2},
\end{equation}
here $c_1=c_1(d^-,d^+)$ is some fixed positive constant. Applying
the generalized Esseen inequality (see for example
\cite{petrov_sums}), we obtain
$$
\sup_{x\in \sym R}\left| F_n(x)-\Phi (x)+{\frac1{\sqrt{2\pi}}}e^{-x^2/2}
          C_n\right|\ll \int_{- T}^T{\frac{\left| \phi_n(t)- e^{-t^2/2}\left( 1+C_nit
\right) \right|}{|t|}}\,dt+\frac1T.
$$
Representing the integral on the right hand side of this
inequality as a sum of integrals over the intervals $|t|\leqslant
\delta L_{n,p}^{-1/p}$ and $\delta L_{n,p}^{-1/p}<|t|\leqslant T$ and
 applying estimates (\ref{phi}) and (\ref{phibound}) in those intervals we obtain
 the proof of the theorem.

\end{proof}

\begin{thebibliography}{10}

\bibitem{babu_manst_zakh_2007}
G.~J. Babu, E.~Manstavi{\v{c}}ius, and V.~Zacharovas.
\newblock Limiting processes with dependent increments for measures on
  symmetric group of permutations.
\newblock In {\em Probability and number theory---{K}anazawa 2005}, volume~49
  of {\em Adv. Stud. Pure Math.}, pages 41--67. Math. Soc. Japan, Tokyo, 2007.

\bibitem{comtet_1974}
L.~Comtet.
\newblock {\em Advanced combinatorics}.
\newblock D. Reidel Publishing Co., Dordrecht, enlarged edition, 1974.
\newblock The art of finite and infinite expansions.

\bibitem{flajolet_odlyzko_1990}
Ph. Flajolet and A.~Odlyzko.
\newblock Singularity analysis of generating functions.
\newblock {\em SIAM J. Discrete Math.}, 3(2):216--240, 1990.

\bibitem{goncharov}
W.~Gontcharoff.
\newblock Sur la distribution des cycles dans les permutations.
\newblock {\em C. R. (Doklady) Acad. Sci. URSS (N.S.)}, 35:267--269, 1942.

\bibitem{halasz_1968}
G.~Hal{\'a}sz.
\newblock \"{U}ber die {M}ittelwerte multiplikativer zahlentheoretischer
  {F}unktionen.
\newblock {\em Acta Math. Acad. Sci. Hungar.}, 19:365--403, 1968.

\bibitem{hardy_divergent}
G.~H. Hardy.
\newblock {\em Divergent {S}eries}.
\newblock Oxford, at the Clarendon Press, 1949.

\bibitem{korevaar_tauberian_book}
J.~Korevaar.
\newblock {\em Tauberian theory}, volume 329 of {\em Grundlehren der
  Mathematischen Wissenschaften [Fundamental Principles of Mathematical
  Sciences]}.
\newblock Springer-Verlag, Berlin, 2004.
\newblock A century of developments.

\bibitem{manst_berry_1998}
E.~Manstavi{\v{c}}ius.
\newblock The {B}erry-{E}sseen bound in the theory of random permutations.
\newblock {\em Ramanujan J.}, 2(1-2):185--199, 1998.

\bibitem{manst_tauber_1999}
E.~Manstavi{\v{c}}ius.
\newblock A {T}auber theorem and multiplicative functions on permutations.
\newblock In {\em Number theory in progress, {V}ol. 2
  ({Z}akopane-{K}o\'scielisko, 1997)}, pages 1025--1038. de Gruyter, Berlin,
  1999.

\bibitem{manst_decomposable_2002}
E.~Manstavi{\v{c}}ius.
\newblock Mappings on decomposable combinatorial structures: analytic approach.
\newblock {\em Combin. Probab. Comput.}, 11(1):61--78, 2002.

\bibitem{manst_additive_1996}
{E}. Manstavi\v{c}ius.
\newblock Additive and multiplicative functions on random permutations.
\newblock {\em Liet. Mat. Rink.}, 36(4):501--511, 1996.

\bibitem{petrov_sums}
V.~V. Petrov.
\newblock {\em Sums of independent random variables}.
\newblock Springer-Verlag, New York, 1975.
\newblock Translated from the Russian by A. A. Brown, Ergebnisse der Mathematik
  und ihrer Grenzgebiete, Band 82.

\bibitem{tauber}
A.~Tauber.
\newblock Ein {S}atz aus der {T}heorie der unendlichen {R}eihen.
\newblock {\em Monatsh. Math. Phys.}, 8(1):273--277, 1897.

\bibitem{zakh_cesaro_2001}
V.~Zacharovas.
\newblock Ces\`aro summation and multiplicative functions on a symmetric group.
\newblock {\em Liet. Mat. Rink.}, 41(Special Issue):140--148, 2001.

\bibitem{zakh_palanga_2001}
V.~Zacharovas.
\newblock The convergence rate in {CLT} for random variables on permutations.
\newblock In {\em Analytic and probabilistic methods in number theory (Palanga,
  2001)}, pages 329--338. TEV, Vilnius, 2002.

\end{thebibliography}

\end{document}